\pdfoutput=1
\documentclass[11pt, a4paper, notitlepage]{article}

\usepackage[left=4.1cm,right=4.1cm,top=3.3cm,bottom=3.3cm]{geometry}
\usepackage[latin1]{inputenc}
\usepackage[english, british]{babel}
\usepackage{calc,amssymb,mathtools,needspace}
\mathtoolsset{mathic}
\usepackage{amsthm}
\usepackage{stmaryrd,units,tabularx,booktabs,paralist,multicol,graphicx,pbox,eqparbox}
\usepackage{hyperref}
\hypersetup{
  breaklinks=true,
  colorlinks=true,
  linkcolor=black,
  anchorcolor=black,
  citecolor=black,
  filecolor=black,
  menucolor=black,
  pagecolor=black,
  urlcolor=black
}
\usepackage[alphabetic]{amsrefs}%non-sorted-cites
\usepackage[all]{xy}
\usepackage[nameinlink]{cleveref}
\chardef\bslash=`\\ % p. 424, TeXbook

\theoremstyle{plain} %% This is the default
\newtheorem{plainnumberthm}{Theorem}\crefname{plainnumberthm}{theorem}{theorems}
\swapnumbers
\newtheorem{thm}{Theorem}[section]\crefname{thm}{theorem}{theorems}
\crefname{mdthm}{theorem}{theorems}
\newtheorem*{thm*}{Theorem}

\newtheorem{cor}[thm]{Corollary}\crefname{cor}{corollary}{corollaries}
\crefname{mdcor}{corollary}{corollaries}
\newtheorem*{cor*}{Corollary}

\newtheorem{lem}[thm]{Lemma}\crefname{lem}{lemma}{lemmas}
\crefname{mdlem}{lemma}{lemmas}
\newtheorem*{lem*}{Lemma}

\newtheorem{prop}[thm]{Proposition}\crefname{prop}{proposition}{propositions}
\crefname{mdprop}{proposition}{propositions}

\crefname{lemSublagrangian}{Sub-Lagrangian Reduction}{}
\newtheorem{lemFiltration}[thm]{Filtration Lemma}\crefname{lemFiltration}{the Filtration Lemma}{}
\crefname{lemGeneration}{the Generation Lemma}{}
\crefname{lemLineGeneration}{the Line Generation Lemma}{}

\theoremstyle{definition}
\newtheorem{defn}[thm]{Definition}\crefname{defn}{definition}{definitions}
\newtheorem*{defn*}{Definition}
\theoremstyle{remark}
\newtheorem{example}[thm]{Example}\crefname{example}{example}{examples}
\newtheorem*{example*}{Example}
\newtheorem{rem}[thm]{Remark}\crefname{rem}{remark}{remarks}
\newtheorem*{rem*}{Remark}
\newtheorem*{warning*}{Warning}
\newtheorem*{hilfslem*}{Lemma}

\renewcommand*{\AA}{\mathbb{A}}
\newcommand*{\ZZ}{\mathbb{Z}}
\newcommand*{\RR}{\mathbb{R}}
\newcommand*{\CC}{\mathbb{C}}

\newcommand*{\QQ}{\mathbb{Q}}
\newcommand*{\HH}{\mathbb{H}}
\newcommand*{\PP}{\mathbb{P}}

\newcommand*{\factor}[2]{\left.\raisebox{.1em}{\ensuremath{#1}}\middle/\raisebox{-.1em}{\ensuremath{#2}}\right.}
\providecommand*{\abs}[1]{\lvert#1\rvert}

\newcommand*{\eps}{\varepsilon}
\newcommand*{\red}{\hat}%reduced elements
\renewcommand*{\bar}{\overline}
\renewcommand*{\tilde}{\widetilde}

\renewcommand*{\setminus}{{-}}
\DeclareMathOperator{\characteristic}{char}

\DeclareMathOperator{\id}{id}
\DeclareMathOperator{\rank}{rank}

\DeclareMathOperator{\im}{im}

\DeclareMathOperator{\coker}{coker}
\DeclareMathOperator{\coeq}{coeq}

\DeclareMathOperator{\Spec}{Spec}
\newcommand*{\Rep}{\mathcal Rep}

\newcommand*{\cc}[1]{#1^\bullet}

\newcommand*{\scup}{\mathbin{\scalebox{0.5}{\ensuremath{\cup}}}}%small cup product symbol

\DeclareMathOperator{\graded}{gr}
\newcommand*{\gr}{\graded}

\newcommand*{\K}{\mathrm K}             
\newcommand*{\W}{\mathrm W}             
\newcommand*{\I}{\mathrm I}             
   
\newcommand*{\GW}{\mathrm{GW}}          
\newcommand*{\GI}{\mathrm{GI}}

\newcommand*{\KO}{{\mathrm{KO}}}
\newcommand{\CH}{{\mathrm{CH}}}
\newcommand{\Pic}{{\mathrm{Pic}}}

\newcommand{\GL}[1]{\mathrm{GL}_{#1}}
\newcommand{\OO}[1]{\mathrm{O}_{#1}}

\renewcommand*{\vec}[1]{\pmb{#1}}
\newcommand*{\cat}[1]{{\mathcal{#1}}}

\newcommand*{\lb}[1]{{\mathcal{#1}}}
\newcommand*{\vb}[1]{{\mathcal{#1}}}
\newcommand*{\sheaf}[1]{{\mathcal{#1}}}
\newcommand*{\dual}{\vee}

\newcommand*{\ur}{\mathit{ur}}
\newcommand*{\classical}{\mathit{clas}}
\renewcommand*{\top}{\mathit{top}}
\newcommand*{\et}{\mathit{et}}
\newcommand*{\nondeg}{\mathit{nd}}
\newcommand*{\gammaF}[1]{{F_{\gamma}^{#1}}}
\newcommand*{\topF}[1]{{F_{\top}^{#1}}}
\newcommand*{\clasF}[1]{{F_{\classical}^{#1}}}
\newcommand*{\urF}[1]{{F_{K}^{#1}}}

\newcommand*{\pretimes}{\mathbin{\hat\times}}
\newcommand*{\tensor}{\mathbin{\otimes}}
\newcommand*{\pretensor}{\mathbin{\hat\tensor}}
\newcommand*{\preLambda}{\smash{\hat\Lambda}}
\newcommand*{\prebigoplus}{\hat\bigoplus}

\newcommand{\sG}{{\sheaf G}}
\newcommand{\sS}{{\sheaf S}}
\renewcommand{\O}{{\sheaf O}}
\newcommand{\sym}{\text{sym}}

\newcommand{\ctext}[2]{\text{\pbox{#1}{\relax\ifvmode\centering\fi #2}}}
\newcommand{\cttext}[2]{\text{\pbox[t]{#1}{\relax\ifvmode\centering\fi #2}}}

\newcommand{\mm}[1]{\left(\begin{smallmatrix}#1\end{smallmatrix}\right)}

\newcolumntype{M}{>{$}c<{$}}
\entrymodifiers={+!!<0pt,\fontdimen22\textfont2>}
\SelectTips{cm}{10}

\newcommand*{\ie}{\mbox{i.\thinspace{}e.\ }}
\newcommand*{\eg}{\mbox{e.\thinspace{}g.\ }}
\newcommand*{\cf}{\mbox{c.\thinspace{}f.\ }}
\newcommand*{\define}[1]{\textbf{#1}}
\newcommand*{\GrothendieckWitt}{Grothen\-dieck-Witt }

\begin{document}
\title{The \(\gamma\)-filtration on the Witt ring of a scheme}
\author{Marcus Zibrowius}
\date{\today}
\maketitle
\begin{abstract}
\noindent
The K-ring of symmetric vector bundles over a scheme \(X\), the so-called Grothendieck-Witt ring of \(X\), can be endowed with the structure of a (special) \(\lambda\)-ring.  The associated \(\gamma\)-filtration generalizes the fundamental filtration on the (Grothendieck-)Witt ring of a field and is closely related to the ``classical'' filtration by the kernels of the first two Stiefel-Whitney classes.
\end{abstract}

\tableofcontents

\section*{Introduction}
In this article, we establish a (special) \(\lambda\)-ring structure on the Grothendieck-Witt ring of a scheme, and some basic properties of the associated \(\gamma\)-filtration.

As far as Witt rings of fields are concerned, there is an unchallenged natural candidate for a good filtration: the ``fundamental filtration'', given by powers of the ``fundamental ideal''.  Its claim to fame is that the associated graded ring is isomorphic to the mod-2 ^^e9tale cohomology ring, as predicted by Milnor \cite{Milnor} and verified by Voevodsky et al.\ \citelist{\cite{Voevodsky:Milnor}\cite{OVV:Milnor}}.  For the Witt ring of a more general variety \(X\), there is no candidate filtration of equal renown. The two most frequently encountered filtrations are:
\begin{compactitem}
\item A short filtration which we will refer to as the {\bf classical filtration}  \(\clasF{*}\W(X)\), given by the whole ring, the kernel of the rank homomorphism, the kernels of the first two Stiefel-Whitney classes. This filtration is used, for example, in \citelist{\cite{Fernandez}\cite{Me:WCS}}.
\item The {\bf unramified filtration} \(\urF{*}\W(X)\), given by the preimage of the fundamental filtration on the Witt ring of the function field \(K\) of \(X\) under the natural homomorphism \(\W(X)\to  \W(K)\).  Said morphism is not generally injective (\eg \cite{Totaro:Witt}), at least not when \(\dim(X) > 3\), and its kernel will clearly be contained in every piece of the filtration.  Recent computations with this filtration include \cite{FunkHoobler}.% but it is in fact much older, see for example \cite{Parimala}.
\end{compactitem}
Clearly, the unramified filtration coincides with the fundamental filtration in the case of a field, and so does the classical filtration as far as it is defined.
The same will be true of the \(\gamma\)-filtration introduced here.  It may be thought of as an attempt to extend the classical filtration to higher degrees.

In general, in order to define a ``\(\gamma\)-filtration'', we simply need to exhibit a pre-\(\lambda\)-structure on the ring in question.  However, the natural candidates for \(\lambda\)-operations, the exterior powers, are not well-defined on the Witt ring \(\W(X)\).  We remedy this by passing to the \GrothendieckWitt ring \(\GW(X)\).  It is defined just like the Witt ring, except that we do not quotient out hyperbolic elements. Consequently, the two rings are related by an exact sequence
\[
\K(X) \to \GW(X) \to \W(X) \to 0.
\]
Given the (pre-)\(\lambda\)-structure on \(\GW(X)\), we can formulate the following theorem concerning the associated \(\gamma\)-filtration.  Let \(X\) be an integral scheme over a field \(k\) of characteristic not two.
\needspace{2cm}
\begin{plainnumberthm}\label{mainthm:filtration}\hfill
  \begin{compactenum}[(1)]
  \item The \(\gamma\)-filtration on \(\GW(k)\) is the  fundamental filtration.
  \item % \(X\) connected scheme over \(k\) (so that \(\gamma\)-filtration can be defined)
    The \(\gamma\)-filtration on \(\GW(X)\) is related to the classical filtration as follows:
    \begin{align*}
      \gammaF{1}\GW(X) &= \clasF{1}\GW(X) := \ker\left(\GW(X)\xrightarrow{\rank}\ZZ\right)\\
      \gammaF{2}\GW(X) &= \clasF{2}\GW(X) :=\ker\left(\clasF{1}\GW(X)\xrightarrow{w_1} H^1_\et(X,\ZZ/2)\right)\\
      \gammaF{3}\GW(X) &\subseteq \clasF{3}\GW(X) :=\ker\left(\clasF{2}\GW(X)\xrightarrow{w_2} H^2_\et(X,\ZZ/2)\right)
    \end{align*}
    However, the inclusion at the third step is not in general an equality.
  \item % \(X\) integral scheme over \(\k\) (so that we have a unique generic point, which we need to define the unramified filtration)
    The \(\gamma\)-filtration on \(\GW(X)\) is finer than the unramified filtration. %\(\gammaF{i}\GW(X)\subset\urF{i}\GW(X)\).
  \end{compactenum}
\end{plainnumberthm}

We define the ``\(\gamma\)-filtration'' on the Witt ring as the image of the above filtration under the canonical projection \(\GW(X)\to\W(X)\).
Thus, each of the above statements easily implies an analogous statement for the Witt ring:  the \(\gamma\)-filtration on the Witt ring of a field is the fundamental filtration, \(\gammaF{i}\W(X)\) agrees with \(\clasF{i}\W(X)\) for \(i<3\) etc.
The same example as for the \GrothendieckWitt ring (\Cref{eg:punctured-A^d}) will show that \(\gammaF{3}\W(X) \neq \clasF{3}\W(X)\) in general.

Most statements of \Cref{mainthm:filtration} also hold under weaker hypotheses---see {\it (1)} \Cref{prop:local-filtration}, {\it (2)} \Cref{comparison:F2,comparison:F2F3} and {\it (3)} \Cref{comparison:finer-than-unramified}. On the other hand, under some additional restrictions, the relation with the unramified filtration can be made more precise.  For example, if \(X\) is a regular variety of dimension at most three and \(k\) is infinite, the unramified filtration on the Witt ring agrees with the global sections of the sheafified \(\gamma\)-filtration (\Cref{sec:unramified-filtration}).

The crucial assertion is of course the equality of \(F^2_\gamma\GW(X)\) with the kernel of \(w_1\)---all other statements would hold similarly for the naive filtration of \(\GW(X)\) by the powers of the ``fundamental ideal'' \(F^1_\gamma\GW(X)\).  The equality follows from the fact that the exterior powers make \(\GW(X)\) not only a pre-\(\lambda\)-ring, but even a \(\lambda\)-ring:\footnote{
  In older terminology, pre-\(\lambda\)-rings are called ``\(\lambda\)-rings'', while \(\lambda\)-rings are referred to as ``special \(\lambda\)-rings''. See also the introduction to \cite{Me:LambdaReps}.}
\begin{plainnumberthm}\label{mainthm}
  For any scheme \(X\) over a field of characteristic not two, the exterior power operations give \(\GW(X)\) the structure of a \(\lambda\)-ring.
\end{plainnumberthm}
In the case when \(X\) is a field, this was established in \cite{McGarraghy:exterior}.  The underlying pre-\(\lambda\)-structure for affine \(X\) has also recently been established independently in \cite{Xie}, where it is used to study sums-of-squares formulas.

Although in this article the \(\lambda\)-structure is used mainly as a tool in proving \Cref{mainthm:filtration}, it should be noted that \(\lambda\)-rings and even pre-\(\lambda\)-rings have strong structural properties.  Much of the general structure of Witt rings of \emph{fields}---for example, the fact they contain no \(p\)-torsion for odd \(p\)---could be (re)derived using the pre-\(\lambda\)-structure on the \GrothendieckWitt ring.  Among the few results that generalize immediately to Grothendieck-Witt rings of schemes is the fact that torsion elements are nilpotent: this is true in any pre-\(\lambda\)-ring.  For \(\lambda\)-rings, Clauwens has even found a sharp bound on the nilpotence degree \cite{Clauwens}.  In our situation, Clauwens result reads:
\begin{cor*}
  Let \(X\) be as above. Suppose \(x\in\GW(X)\) is an element satisfying \(p^e x = 0\) for some prime \(p\) and some exponent \(e>0\).  Then \mbox{\(x^{p^e+p^{e-1}} = 0\)}.
\end{cor*}
To put the corollary into context, recall that for a field \(k\) of characteristic not two,  an element \(x\in\GW(k)\) is nilpotent if and only if it is \(2^n\)-torsion for some \(n\) \citelist{\cite{Lam}*{VIII.8}\cite{Me:App}}.  This equivalence may be generalized at least to connected semi-local rings in which two is invertible, using the pre-\(\lambda\)-structure for one implication and \cite{KRW72}*{Ex.~3.11} for the other.  See \cite{McGarraghy:exterior} for further applications of the \(\lambda\)-ring structure on \GrothendieckWitt rings of fields and \cite{Balmer:Nilpotence} for nilpotence results for Witt rings of regular schemes.

From the \(\lambda\)-theoretic point of view, the main complication in the \GrothendieckWitt ring of a general scheme as opposed to that of a field is that not all generators can be written as sums of line elements. In K-theory, this difficulty can often be overcome by embedding \(\K(X)\) into the K-ring of some auxiliary scheme in which a given generator does have this property, but in our situation this is impossible:  there is no splitting principle for \GrothendieckWitt rings (\Cref{sec:no splitting}).

\subsection*{Acknowledgements}
I thank Pierre Guillot for getting me started on these questions, and Kirsten Wickelgren for making me aware of the above-mentioned nilpotence results.

\section{Generalities}
\subsection{$\lambda$-rings}\label{sec:general-lambda}
\newcommand{\symGroup}{\mathfrak S}
\newcommand{\symFunctions}{\mathbb W}
\newcommand{\ev}{\mathrm{ev}}
We give a quick and informal introduction to \(\lambda\)-rings, treading medium ground between the traditional definition in terms of exterior power operations \cite{SGA6}*{Expos^^e9~V}  and the abstract definition of \(\lambda\)-rings as coalgebras over a comonad \cite{Borger:BasicI}*{1.17}.  The main point we would like to get across is that a \(\lambda\)-ring is ``a ring equipped with all possible symmetric operations'', not just ``a ring with exterior powers''.  This observation is not essential for anything that follows---we will later work exclusively with the traditional definition---but we hope that it provides some intrinsic motivation for considering this kind of structure.

To make our statement more precise, let \(\symFunctions\) be the ring of symmetric functions.  That is, \(\symFunctions\) consists of all formal power series \(\phi(x_1,x_2,\dots)\) in countably many variables \(x_1,x_2,\dots\) with coefficients in \(\ZZ\) such that \(\phi\) has bounded degree and such that the image \(\phi(x_1,\dots,x_n,0,0,\dots)\) of \(\phi\) under the projection to \(\ZZ[x_1,\dots,x_n]\) is a symmetric polynomial for all \(n\). For example, \(\symFunctions\) contains~\dots
\begin{compactitem}[\dots]
  \item \eqparbox{egfunctions}{the {\bf elementary symmetric functions}} \(\lambda_k := \sum_{i_1<\dots<i_k} x_{i_1}\cdot\ldots\cdot x_{i_k}\),
  \item \eqparbox{egfunctions}{the {\bf complete symmetric functions}} \(\sigma_k := \sum_{i_1\leq\dots\leq i_k} x_{i_1}\cdot\ldots\cdot x_{i_k}\),
  \item \eqparbox{egfunctions}{the {\bf Adams symmetric functions}} \(\psi_k := \sum_{i} x_i^k\).
\end{compactitem}
The first two families of symmetric functions each define a set of algebraically independent generators of \(\symFunctions\) over \(\ZZ\), so they can be used to identify \(\symFunctions\) with a polynomial ring in countably many variables. Another, equivalent set of generators is given by the so-called  Witt symmetric functions (\cite{Borger:Positivity}*{4.5}).  The Adams symmetric functions are also algebraically independent, but they only generate \(\symFunctions\otimes_{\ZZ}\QQ\) over \(\QQ\).  In any case, we have no need to choose any specific set of generators just now.

Given a commutative ring \(A\), we write \(\symFunctions A\) for the {\bf universal \(\lambda\)-ring}\footnote{
  This ring is also known as the big Witt ring with coefficients in~\(A\), or as the big ring of Witt vectors over~\(A\).  We avoid this terminology here.  While the ``Witt'' in ``Witt vectors'' and the ``Witt'' in ``Witt ring of quadratic forms'' refer to the same Ernst Witt, these are otherwise fairly independent concepts.}
{\bf over \(A\)}.
As a set, it consists of all ring homomorphisms from \(\symFunctions\) to \(A\):
\[
\symFunctions A = \cat Rings(\symFunctions, A).
\]
In particular, for every symmetric function \(\phi\in\symFunctions\), we have an evaluation map \(\ev_\phi\colon \symFunctions A\to A\).  The universal \(\lambda\)-ring \(\symFunctions A\) becomes a ring via a coproduct \(\Delta^+\) and a comultiplication \(\Delta^\times\) on \(\symFunctions\).  (These cooperations are determined by the equations \(\Delta^+(\psi_n) = 1\otimes \psi_n + \psi_n\otimes 1\) and \(\Delta^\times(\psi_n) = \psi_n\otimes\psi_n\) for all~\(n\) \cite{Borger:Positivity}*{Introduction}.)

\begin{defn}
  A {\bf pre-\(\lambda\)-ring} is a commutative ring \(A\) together with a group homomorphism \(\theta_A\colon A\to \symFunctions A\) such that
  \[\xymatrix{
    {A} \ar@{=}[dr] \ar[r]^{\theta_A} & {\symFunctions A} \ar[d]^{\ev_{\psi_1}}\\
    & {A}
  }\]
  commutes. A morphism of pre-\(\lambda\)-rings \((A,\theta_A)\to (B,\theta_B)\) is a ring homomorphism \(f\colon A\to B\) such that \(\symFunctions(f)\theta_A = \theta_B f\).  We refer to such a morphism as a \define{\(\lambda\)-morphism}.
\end{defn}
It would, of course, appear more natural to ask for the map \(\theta_A\) to be a \emph{ring} homomorphism.  But this requirement is only one of the two additional requirements reserved for \(\lambda\)-rings.  The second additional requirement takes into account that the universal \(\lambda\)-ring \(\symFunctions A\) can itself be equipped with a canonical pre-\(\lambda\)-structure for any ring \(A\).

\begin{defn}
  A {\bf \(\lambda\)-ring} is a pre-\(\lambda\)-ring \((A,\theta_A)\) such that \(\theta_A\) is a \(\lambda\)-morphism.
\end{defn}
It turns out that the canonical pre-\(\lambda\)-structure does make the universal \(\lambda\)-ring \(\symFunctions A\) a \(\lambda\)-ring, so the terminology is sane.  The observation alluded to at the beginning of this section is that any symmetric function \(\phi\in\symFunctions\) defines an ``operation'' on any (pre-)\(\lambda\)-ring \(A\), \ie a map \(A\to A\):  the composition of \(\ev_{\phi}\) with \(\theta_A\).
\[\xymatrix{
  {A} \ar@{..>}[dr]\ar[r]^{\theta_A} & {\symFunctions A} \ar[d]^{\ev_{\phi}} \\
  & A
  }
\]
In particular, we have families of operations \(\lambda^k\), \(\sigma^k\) and \(\psi^k\) corresponding to the symmetric functions specified above.  They are referred to as {\bf exterior power operations}, {\bf symmetric power operations} and {\bf Adams operations}, respectively.

The underlying additive group of the universal \(\lambda\)-ring \(\symFunctions A\) is isomorphic to the multiplicative group \((1 + t A\llbracket t \rrbracket)^\times\) inside the ring of invertible power series over \(A\), and the isomorphism can be chosen such that the projection onto the coefficient of \(t^i\) corresponds to \(\ev_{\lambda_i}\) (e.g.~\cite{Hesselholt:Big}*{Prop.~1.14, Rem.~1.21(2)}).  Thus, a pre-\(\lambda\)-structure is completely determined by the operations \(\lambda^i\), and conversely, any family of operations \(\lambda^i\) for which the map
\begin{align*}
A &\xrightarrow{\;\;\lambda_t\;} (1 + t A\llbracket t \rrbracket)^\times\\
a &\;\;\mapsto\;\; 1 + \lambda^1(a)t + \lambda^2(a)t^2 + \dots
\end{align*}
is a group homomorphism, and for which \(\lambda^1(a) = a\), defines a \(\lambda\)-structure. This recovers the traditional definition of a pre-\(\lambda\)-structure as a family of operations \(\lambda^i\) (with \(\lambda^0=1\) and \(\lambda^1=\id\)) satisfying the relation \(\lambda^k(x+y) = \sum_{i+j=k}\lambda^i(x)\lambda^j(y)\) for all \(k\geq 0\) and all \(x,y\in A\).

The question whether the resulting pre-\(\lambda\)-structure is a \(\lambda\)-structure can similarly be reduced to certain polynomial identities, though these are more difficult to state and often also more difficult to verify in practice.  However, for pre-\(\lambda\)-rings with some additional structure, there are certain standard criteria that make life easier.
\begin{defn}
  An \define{augmented (pre-)\(\lambda\)-ring} is a (pre-)\(\lambda\)-ring \(A\) together with a \(\lambda\)-morphism
  \[
  d\colon A\to \ZZ,
  \]
  where the (pre-)\(\lambda\)-structure on \(\ZZ\) is defined by \(\lambda^i(n) := \binom{n}{i}\).

  A \define{(pre-)\(\lambda\)-ring with positive structure} is an augmented (pre-)\(\lambda\)-ring \(A\) together with a specified subset \(A_{>0}\subset A\) on which \(d\) is positive and which generates \(A\) in the strong sense that any element of \(A\) can be written as a difference of elements in \(A_{>0}\); it is moreover required to satisfy a list of axioms for which we refer to \cite{Me:LambdaReps}*{\S3}.
\end{defn}

For example, one of the axioms for a positive structure is that for an element \(e\in A_{>0}\), the exterior powers \(\lambda^k e\) vanish for all \(k>d(e)\).  We will refer to elements of \(A_{>0}\) as \define{positive elements}, and to positive elements \(l\) of augmentation \(d(l) = 1\) as \define{line elements}.  The motivating example, the K-ring \(\K(X)\) of a connected scheme \(X\), is augmented by the rank homomorphism, and a set of positive elements is given by the classes of vector bundles.  The situation for the \GrothendieckWitt ring will be analogous.

Here are two simple criteria for showing that a pre-\(\lambda\)-ring with positive structure is a \(\lambda\)-ring:
\begin{description}
\item[Splitting Criterion] If all positive elements of \(A\) decompose into sums of line elements, then \(A\) is a \(\lambda\)-ring.

\item[Detection Criterion] If for any pair of positive elements \(e_1, e_2 \in A_{>0}\) we can find a \(\lambda\)-ring \(A'\) and a \(\lambda\)-morphism \(A'\to A\) with both \(e_1\) and \(e_2\) in its image, then \(A\) is a \(\lambda\)-ring.
\end{description}
We again refer to \cite{Me:LambdaReps} for details.

\subsection{The $\gamma$-filtration}
The \define{\(\gamma\)-operations} on a pre-\(\lambda\)-ring \(A\) can be defined as \(\gamma^n(x) := \lambda^n(x+n-1)\).  They again satisfy the identity \(\gamma^k(x+y) = \sum_{i+j=k}\gamma^i(x)\gamma^j(y)\).
\begin{defn}
  The \(\gamma\)-filtration on an augmented pre-\(\lambda\)-ring \(A\) is defined as follows:
  \begin{align*}
    \gammaF{0} A &:= A\\
    \gammaF{1} A &:= \ker(A\xrightarrow{d} \ZZ)\\
    \gammaF{i} A &:= \left(\ctext{10cm}{subgroup generated by all finite products\\ \(\textstyle\prod_j \gamma^{i_j}(a_j)\) with \(a_j\in \gammaF{1} A\) and \(\textstyle\sum_j{i_j} \geq i\)}\right)
    &\text{ for } i > 1
  \end{align*}
\end{defn}
This is in fact a filtration by ideals, multiplicative in the sense that \(\gammaF{i} A \cdot \gammaF{j} A \subset \gammaF{i+j}A\), hence we have an associated graded ring
\[
\gr^*_\gamma A := \bigoplus_i \gammaF{i} A/ \gammaF{i+1} A.
\]
See \cite{AtiyahTall}*{\S4} or \cite{FultonLang}*{III~\S1} for details.  The following lemma is sometimes useful for concrete computations.
\begin{lem}\label{lem:gamma-filtration-generators}
    If \(A\) is a pre-\(\lambda\)-ring with positive structure such that every positive element in \(A\) can be written as a sum of line elements, then \(\gammaF{k} A = (\gammaF{1} A)^k\).

  More generally, suppose that \(A\) is an augmented pre-\(\lambda\)-ring, and let \(E\subset A\) be some set of additive generators of \(\gammaF{1} A\).
  Then \(\gammaF{k} A\) is additively generated by finite products of the form
  \(
  \prod_j \gamma^{i_j}(e_j)
  \)
  with \(e_j\in E\) and \(\sum_j i_j \geq k\).
\end{lem}
\begin{proof}
  The first assertion may be found in \cite{FultonLang}*{III~\S1}.
  It also follows from the second, which we now prove.
  As each \(x\in \gammaF{1} A\) can be written as a linear combination of elements of \(E\),  we can write any \(\gamma^{i}(x)\) as a linear combination of products of the form \(\prod_j \gamma^{i_j}(\pm e_j)\) with \(e_j\in E\) and \(\sum_j i_j = i\).
  Thus, \(\gammaF{k} A\) can be generated by finite products of the form \(\prod_j \gamma^{i_j}(\pm e_j)\), with \(e_j\in E\) and \(\sum_j i_j\geq k\).
  Moreover, \(\gamma^i(-e)\) is a linear combination of products of the form \(\prod_j \gamma^{i_j}(e)\) with \(\sum_j i_j = i\):  this follows from the above identity for \(\gamma^k(x+y)\).
  Thus, \(\gammaF{k} A\) is already generated by products of the form described.
\end{proof}

For \(\lambda\)-rings with positive structure, we also have the following general fact:
\begin{lem}[\cite{FultonLang}*{III, Thm~1.7}]\label{lem:graded-degree-1}
  For any \(\lambda\)-ring \(A\) with positive structure, the additive group \(\gr^1 A = \gammaF{1} A/\gammaF{2} A\) is isomorphic to the multiplicative group of line elements in \(A\).
\end{lem}

\section{The $\lambda$-structure on the Grothendieck-Witt ring}
\subsection{The pre-$\lambda$-structure}
\begin{prop}\label{prop:lambda-on-GW}
  Let \(X\) be a scheme. The exterior power operations  \( \lambda^k\colon (\vb M, \mu) \mapsto (\Lambda^k \vb M, \Lambda^k \mu) \) induce well-defined maps on \(\GW(X)\) which provide \(\GW(X)\) with the structure of a pre-\(\lambda\)-ring.
\end{prop}
Our proof of the existence of a pre-\(\lambda\)-structure will follow the same pattern as the proof for symmetric representation rings in \cite{Me:LambdaReps}:

\paragraph{Step 1.} The assignment \(\lambda^i(\vb M, \mu) := (\Lambda^i \vb M, \Lambda^i\mu)\) is well-defined on the set of isometry classes of symmetric vector bundles over \(X\), so that we have an induced map
\begin{align*}
\lambda_t\colon \left\{\ctext{5cm}{isometry classes of symmetric vector bundles over \(X\)}\right\} &\longrightarrow (1 + t \GW(X)\llbracket t \rrbracket)^\times.
\intertext{%
We extend it linearly to a group homomorphism
}
\bigoplus\ZZ(\vb M, \mu) &\longrightarrow (1 + t \GW(X)\llbracket t \rrbracket)^\times,
\end{align*}
where the sum on the left is over all isometry classes of symmetric vector bundles over \(X\).

\paragraph{Step 2.} The map \(\lambda_t\) is additive in the sense that
\[
\lambda_t((\vb M, \mu)\perp(\vb N,\nu)) = \lambda_t(\vb M, \mu)\lambda_t(\vb N,\nu).
\]
Thus, it factors through the quotient of \(\bigoplus\ZZ(\vb M, \mu)\) by the ideal generated by the relations
\(
((\vb M, \mu)\perp(\vb N,\nu)) = (\vb M, \mu) + (\vb N,\nu)
\).
\paragraph{Step 3.}  The homomorphism \(\lambda_t\) respects the relation \( (\vb M,\mu) = H(\vb L)\) for every metabolic vector bundle \( (\vb M,\mu) \) with Lagrangian \(\vb L\). Thus, we obtain the desired factorization
\[
\lambda_t\colon \GW(X)\to (1 + t \GW(X)\llbracket t \rrbracket)^\times.
\]

To carry out these steps, we only need to replace all arguments on the level of vector spaces of \cite{Me:LambdaReps} with local arguments.  We formulate the key lemma in detail and then sketch the remaining part of the proof.
\begin{lemFiltration}[\cf~\cite{SGA6}*{Expos^^e9~V, Lemme~2.2.1}]\label{lemFiltration}
  Let \(0\to \vb L\to \vb M\to \vb N\to 0\) be an extension of vector bundles over a scheme \(X\). Then we can find a filtration of \(\Lambda^n \vb M\) by sub-vector bundles
  \(
  \Lambda^n \vb M = \vb M^0 \supset \vb M^1 \supset \vb M^2 \supset \cdots
  \)
  together with isomorphisms
  \[
    \pi_{\vb M}^i\colon \factor{\vb M^i}{\vb M^{i+1}} \cong \Lambda^i \vb L \otimes \Lambda^{n-i} \vb N.
  \]
  More precisely, there is a unique way of associating such filtrations and isomorphisms with extensions of vector bundles on schemes subject to the following conditions:
  \begin{itemize}
  \item [(1)]   The filtrations are natural with respect to isomorphisms of extensions. That is, given extensions \(\vb M\) and \(\tilde{\vb M}\) of  \(\vb N\) by \(\vb L\), any  isomorphism \(\phi\colon \vb M\to \tilde{\vb M}\) for which
    \[\xymatrix{
      0 \ar[r] & \vb L \ar[r]\ar@{=}[d] & \vb M          \ar[r]\ar[d]^{\phi}_{\cong} & \vb N \ar[r]\ar@{=}[d] & 0 \\
      0 \ar[r] & \vb L \ar[r]           & {\tilde{\vb M}} \ar[r]                    & \vb N \ar[r]           & 0
    }\]
    commutes restricts to isomorphisms \(\vb M^i \to \tilde{\vb M}^i\) compatible with the isomorphisms \(\pi_{\vb M}^i\) and \(\tilde\pi_{\vb M}^i\) in the sense that
    \[\xymatrix{
      {\factor{\vb M^i}{\vb M^{i+1}}} \ar[rd]^{\cong}_{\pi_{\vb M}^i}\ar[rr]_{\cong}^{\bar \phi} && \ar[ld]_{\cong}^{\pi_{\tilde{\vb M}}^i} {\factor{\tilde{\vb M}^i}{\tilde{\vb M}^{i+1}}} \\
      & {\Lambda^i \vb L\otimes \Lambda^{n-i} \vb N}
    }\]
    commutes.

  \item[(1')]\label{i:filt-1}
    The filtration is natural with respect to morphisms of schemes \mbox{\(f\colon Y\to X\)}:  \( (f^*\vb M)^i = f^*(\vb M^i)\) and \(f^*(\pi_{\vb M}^i) = \pi_{f^*\vb M}^i\) under the identification of  \(\Lambda^n(f^* \vb M)\) with \(f^*(\Lambda^n \vb M)\).

  \item[(2)]\label{i:filt-2}
    For the trivial extension, \((\vb L\oplus \vb N)^i  \subset \Lambda^n(\vb L\oplus \vb N)\) corresponds to the submodule
    \[
    \textstyle\bigoplus_{j\geq i} \Lambda^j \vb L\otimes \Lambda^{n-j} \vb N \quad \subset \quad \textstyle \bigoplus_j \Lambda^j \vb L\otimes \Lambda^{n-j} \vb N
    \]
    under the canonical isomorphism \(\Lambda^n(\vb L\oplus \vb N) \cong \bigoplus_j \Lambda^j \vb L\otimes \Lambda^{n-j} \vb N\), and the iso\-morphisms
    \[
    \pi_{\vb L \oplus \vb N}^i\colon \factor{(\vb L\oplus \vb N)^i}{(\vb L\oplus \vb N)^{i+1}} \xrightarrow{\cong} \Lambda^i \vb L\otimes \Lambda^{n-i} \vb N
    \]
    are induced by the canonical projections.
  \end{itemize}
\end{lemFiltration}
(The numbering {\it (1), (1'), (2)} is chosen to be as close as possible to the Filtration Lemma in \cite{Me:LambdaReps}.  For a closer analogy, statement~{\it (1)} there should also be split into two parts {\it (1)} and {\it (1')}: naturality with respect to isomorphisms of extensions, and naturality with respect to pullback along a group homomorphism.  As stated there, only the pullback to the trivial group is covered.)

\begin{proof}[Proof of \cref{lemFiltration}]
  Uniqueness is clear: if filtrations and isomorphisms satisfying the above conditions exist, they are determined locally by {\it (1)} and {\it (2)}, hence globally by {\it (1')}.

  Existence may be proved via the following direct construction.  Let \(0\to \vb L\xrightarrow{\iota} \vb M \xrightarrow{\pi} \vb N \to 0\) be an arbitrary short exact sequence of vector bundles over \(X\).  Consider the morphism \(\Lambda^i \vb L\otimes \Lambda^{n-i} \vb M \to \Lambda^n \vb M\) induced by \(\iota\).  Let \(\vb M_i\) be its kernel and \(\vb M^i\) its image, so that we have a short exact sequence of quasi\-coherent sheaves:
  \begin{equation*}%\label{eq:lemFiltration:proof}
    \xymatrix{
      0  \ar[r] &  {\vb M_i} \ar[r] & {\Lambda^i \vb L\otimes \Lambda^{n-i} \vb M} \ar[r] & {\vb M^i} \ar[r] & 0
    }
  \end{equation*}
  We claim {\it (a)} that the sheaves \(\vb M_i\) and \(\vb M^i\) are again vector bundles, {\it (b)} that the morphism \(\Lambda^i \vb L\otimes \Lambda^{n-i} \vb M \to \Lambda^i \vb L\otimes \Lambda^{n-i} \vb N\) induced by \(\pi\) factors through \(\vb M^i\) and induces an isomorphism
  \[
  \pi_{\vb M}^i\colon \factor{\vb M^i}{\vb M^{i+1}}\to \Lambda^i \vb L\otimes \Lambda^{n-i} \vb N,
  \]
  and {\it (c)} that this construction of subbundles \(\vb M^i\) and isomorphisms \(\pi_{\vb M}^i\) satisfies the properties {\it (1), (1'), (2)}.
  This can be checked in the following order:  First, verify most of {\it (c)}: the first half of statement~{\it (1)}, statement~{\it (1')} and statement~{\it (2)}.  Then {\it (a)} and {\it (b)} follow because any extension of vector bundles is locally split.  Lastly, the commutativity of the triangle in {\it (1)} can also be checked locally.
\end{proof}

\begin{proof}[Proof of \Cref{prop:lambda-on-GW}]
For Step~1, we note that the exterior power operation
\(
  \Lambda^i\colon  \cat Vec(X) \to \cat Vec(X)
\)
is a duality functor in the sense that we have an isomorphism
\(
   \eta_{\vb M}\colon \Lambda^i(\vb M^\vee) \xrightarrow{\cong} (\Lambda^i \vb M)^\vee
\)
for any vector bundle \(\vb M\). Indeed, we can define \(\eta_{\vb M}\) on sections by
\[\phi_1\wedge\dots\wedge\phi_i \mapsto (m_1\wedge\dots\wedge m_i \mapsto \det(\phi_\alpha(m_\beta)). \]
The fact that this is an isomorphism can be checked locally and follows from \cite{Bourbaki:Algebre}*{Ch.~3, \S\,11.5, (30 bis)}.
We therefore obtain a well-defined operation on the set of isometry classes of symmetric vector bundles over \(X\) by defining
\[
\lambda^i(\vb M,\mu) := (\Lambda^i \vb M, \eta_\vb M\circ \Lambda^i(\mu)).
\]

Step~2 is completely analogous to the argument in \cite{Me:LambdaReps}.

For Step~3, let \( (\vb M,\mu) \) be metabolic with Lagrangian \(\vb L\), so that we have a short exact sequence
  \begin{equation}\label{eq:lambda-on-GW:ses}
  0 \to \vb L \xrightarrow{i} \vb M \xrightarrow{i^\dual\mu} \vb L^\dual \to 0.
  \end{equation}
  When \(n\) is odd, say \(n=2k-1\), we claim that \(\vb M^k\) is a Lagrangian of \(\Lambda^n(\vb M,\mu)\).
  When \(n\) is even, say \(n=2k\), we claim that \(\vb M^{k+1}\) is an admissible sub-Lagrangian of \(\Lambda^n(\vb M,\mu)\) with \((\vb M^{k+1})^\perp = \vb M^k\), and that the composition of isomorphisms
  \[
  \factor{\vb M^{k}}{\vb M^{k+1}} \cong \Lambda^k \vb L\otimes \Lambda^k(\vb L^\dual) \cong \factor{H(\vb L)^{k}}{H(\vb L)^{k+1}}
  \]
  of \cref{lemFiltration} is even an isometry between \((\vb M^{k}/\vb M^{k+1},\mu)\) and \((H(\vb L)^{k}/H(\vb L)^{k+1}, \bar{\mm{0 & 1 \\ 1 & 0}})\).  All of these claims can be checked locally.  Both the local arguments and the conclusions are analogous to those of \cite{Me:LambdaReps}.
\end{proof}

\begin{rem}
  In many cases, Step~3 can be simplified.  If \(X\) is affine, then Step~3 is redundant because any short exact sequence of vector bundles splits.  If \(X\) is a regular quasiprojective variety over a field, we can reduce the argument to the affine case using Jouanolou's trick and homotopy invariance.  Indeed,  Jouanolou's trick yields an affine vector bundle torsor \(\pi\colon E\to X\) \cite{Weibel:KH}*{Prop.~4.3}, and as \(X\) is regular, \(\pi^*\) is an isomorphism: it is an isomorphism on \(\K(-)\) by \cite{Weibel:KH}*{below Ex.~4.7}, an isomorphism on \(\W(-)\) by \cite{Gille:HomotopyInvariance}*{Cor.~4.2}, and hence an isomorphism on \(\GW(-)\) by Karoubi induction.
\end{rem}

\subsection{The pre-$\lambda$-structure is a $\lambda$-structure}
We would now like to show that the pre-\(\lambda\)-structure on \(\GW(X)\) discussed above is an actual \(\lambda\)-structure.  In some cases, this is easy:

\begin{prop}\label{GW-R-special}
  For any connected semi-local commutative ring \(R\) in which two is invertible, the \GrothendieckWitt ring \(\GW(R)\) is a \(\lambda\)-ring.
\end{prop}
\begin{proof}
  Over such a ring, any symmetric space decomposes into a sum of line elements \cite{Baeza}*{Prop.~I.3.4\slash 3.5}, so the result follows from the splitting criterion (\Cref{sec:general-lambda}).
\end{proof}

The following result is more interesting:
\begin{thm}
  For any connected scheme \(X\) over a field of characteristic not two, the pre-\(\lambda\)-structure on \(\GW(X)\) introduced above is a \(\lambda\)-structure.
\end{thm}

The usual strategy for proving the analogous result in K-theory is via a ``geometric splitting principle'': see \Cref{sec:no splitting}.  However, as we will see there, no such principle is available for the \GrothendieckWitt ring.  So instead, we follow an alternative strategy, which we recall from \cite{SGA6}*{Expos^^e9~VI, Thm~3.3}.\footnote{
    The result of Serre invoked at the end of the proof in \cite{SGA6} to verify (K1) is \cite{Serre}*{Thm~4}.}
Let \(G\) be a linear algebraic group scheme over \(k\). The principal components of this alternative strategy are:
\begin{itemize}
\item[(K1)] The representation ring \(\K(\Rep(G))\) is a \(\lambda\)-ring.
\item[(K2)] For any \(G\)-torsor \(\sS\) over \(X\), the map
  \[\K(\Rep(G))\to \K(X)\]
  sending a representation \(V\) of \(G\) to the vector bundle \(\sS\times_{G}V\) is a \(\lambda\)-morphism.  (Here, the \(V\) in \(\sS\times_{G}V\) is to be interpreted as a trivial vector bundle over \(X\).)
\item[(K3)]
 For any pair of vector bundles \(\vb E\) and \(\vb F\) over \(X\), there exists a linear algebraic group scheme \(G\) and a \(G\)-torsor \(\sS\) such that both \(\vb E\) and \(\vb F\) lie in the image of the morphism \(\K(\Rep(G))\to \K(X)\) defined by \(\sS\). (\(G\) can be chosen to be a product of general linear groups.)
\end{itemize}
From these three points, the fact that \(\K(X)\) is a \(\lambda\)-ring follows via the detection criterion (\Cref{sec:general-lambda}).
The same argument will clearly work for \(\GW(X)\) provided the following three analogous statements hold. (We now assume that \(\characteristic k \neq 2\).)
\begin{itemize}
\item[(GW1)] The symmetric representation ring \(\GW(\Rep(G))\) is a \(\lambda\)-ring.
\item[(GW2)] For any \(G\)-torsor \(\sS\) over \(X\), the map
  \[\GW(\Rep(G))\to \GW(X)\]
  sending a symmetric representation \(V\) of \(G\) to the symmetric vector bundle \(\sS\times_{G} V\) is a \(\lambda\)-morphism.
\item[(GW3)]
  Any pair of symmetric vector bundles lies in the image of some common morphism \(\GW(\Rep(G))\to \GW(X)\) defined by a \(G\)-torsor as above. (\(G\) can be chosen to be a product of split orthogonal groups.)
\end{itemize}
Statement (GW1) is the main result of \cite{Me:LambdaReps}.  The remaining points (GW2) and (GW3) are discussed below:  see \Cref{lem:K2GW2} and \Cref{lem:GW3}.  Any reader to whom (K2) and (K3) are obvious will most likely consider (GW2) and (GW3) equally obvious---the only point of note is that for (GW3) we have to work in the ^^e9tale topology rather than in the Zariski topology.

\subsubsection*{Twisting by torsors}
Let \(X\) be a scheme with structure sheaf \(\O\).  We fix some Grothendieck topology on \(X\).  (For (K2) and (GW2), the topology is irrelevant.  For (K3) we can take any topology at least as fine as the Zariski topology, while for (GW3) we will need the ^^e9tale topology.)

Given a sheaf of (not necessarily abelian) groups \(\sG\) over \(X\), recall that a (right) \define{\(\sG\)-torsor} is a sheaf of sets \(\sS\) over \(X\) with a right \(\sG\)-action such that:
\begin{compactenum}[(1)]
\item There exists a cover \(\{U_i\to X\}_i\) such that \(\sS(U_i)\neq \emptyset\) for all \(U_i\).  (Any such cover is said to \define{split} \(\sS\).)
\item For all open \((U\to X)\), and for one (hence for all) \(s\in\sS(U)\), the map
  \begin{align*}
    \sG_{|U} &\xrightarrow{} \sS_{|U}\\
    g &\mapsto s.g
  \end{align*}
  is an isomorphism.
\end{compactenum}

\newcommand{\p}[2]{{_{#1}}{[#2]}} % notation for elements of a direct summand that have only one non-zero coordinate (#1)

\begin{defn}\label{def1}
  Let \(\sS\) be a \(\sG\)-torsor as above. For any presheaf  \(\sheaf E\) of \(\O\)-modules with an \(\O\)-linear left \(\sG\)-action, we define a new presheaf of \(\O\)-modules by
  \begin{align*}
    \sS \pretimes_{\sG} \sheaf E
    &:= \coeq\big( (\sS\times \sG)\tensor \sheaf E \rightrightarrows \sS\tensor \sheaf E \big)\\
    &\,= \coker\big( \textstyle\bigoplus_{\sS,\sG} \sheaf E \longrightarrow \textstyle\bigoplus_{\sS} \sheaf E \big),
  \end{align*}
  where, on any open \(U\), the morphism in the second line has the form \(\p{s,g}{v}\mapsto \p{s.g}{v}-\p{s}{g.v}\) for \(s\in\sS(U)\), \(g\in\sG(U)\) and \(v\in\sheaf E(U)\).\footnote{
    We use square brackets with a subscript on the left for an element of a direct sum that is concentrated in a single summand.  A general element of \(\bigoplus_{s\in S} X_s\)  is a finite sum of the form \(\displaystyle \sum_{s\in S} \p{s}{x_s}\)  in this notation.}
  If \(\sheaf E\) is a sheaf, we  define \(\sS\times_{\sG} \sheaf E\) as the sheafification of \(\sS\pretimes_{\sG}\sheaf E\).  Equivalently, we may define \(\sS\times_{\sG}\sheaf E\) by the same formula as \(\sS\pretimes_{\sG}\sheaf E\) provided we interpret the direct sum and cokernel as direct sum and cokernel in the category of sheaves of \(\O\)-modules.
\end{defn}

\begin{rem}
  The sheaf of \(\O\)-modules \(\sS\times_{\sG}\sheaf E\) can alternatively be described as follows:
  Fix a cover \(\{U_i\to X\}_i\) which splits \(\sS\), and fix an element \(s_i\in S(U_i)\) for each \(U_i\).
  Let \(g_{ij}\in\sG(U_i\times_X U_j)\) be the unique element satisfying \(s_j = s_i.g_{ij}\).
  Then \(\sS\times_{\sG}\sheaf E\) is isomorphic to the sheaf given on any open \((V\to X)\) by
  \[
  \{(v_i)_i \in \textstyle\prod_i\sheaf E(V_i) \mid v_i = g_{ij}.v_j \text{ on } V_{ij} \},
  \]
  where \(V_i := V\times_X U_i\), \(V_{ij} := V_i\times_X V_j\).
  (We will not use this description in the following.)
\end{rem}

We next recall the basic properties of \(\sS\times_{\sG}\sheaf E\).  We call two presheaves of \(\O\)-modules {\bf locally isomorphic} if \(X\) has a cover such that the restrictions to each open of the cover are isomorphic.
A morphism of presheaves of \(\O\)-modules is said to be {\bf locally an isomorphism} if \(X\) has a cover such that the restriction of the morphism to each open of the cover is an isomorphism.

\begin{lem}\label{lem:local-iso}\mbox{}\hfill
  \begin{compactenum}[(i)]
  \item For any presheaf \(\sheaf E\) as above, \(\sS\pretimes_{\sG}\sheaf E\) is locally isomorphic to \(\sheaf E\).
  \item For any sheaf \(\sheaf E\) as above, \(\sS\times_{\sG}\sheaf E\) is locally isomorphic to \(\sheaf E\).
  \item The canonical morphism \(\sS\pretimes_{\sG}\sheaf E\to\sS\times_{\sG}\sheaf E\) is locally an isomorphism for any sheaf as above.
  \end{compactenum}
  More precisely, the presheaves in {\it (i) \& (ii)} are isomorphic over any open \((U\to X)\) such that \(\sS(U)\neq \emptyset\), and likewise the morphism in {\it (iii)} is an isomorphism over any such \(U\).
\end{lem}

\begin{proof}
  For {\it (i)} of the lemma, let \((U\to X)\) be an open such that \(\sS(U)\neq\emptyset\). Fix any \(s\in \sS(U)\). For each \((V\to U)\) and each \(t\in\sS(V)\), there exists a unique element \(g_t\in\sG(V)\) such that \(t = s.g_t\). Therefore, the morphism  \(\bigoplus_{\sS_{|U}} \sheaf E_{|U} \to \sheaf E_{|U}\) sending \(\p{t}{v}\) to \(g_t.v\) describes the cokernel defining \(\sS\pretimes_{\sG}\sheaf E\) over \(U\).
    Statements {\it (ii) \& (iii)} of the lemma follow from {\it (i)}.
\end{proof}

\begin{lem}\label{lem:basic-properties}\mbox{}\hfill
  \begin{compactenum}[(i)]
  \item  The functor \(\sS\times_{\sG} -\) is exact, \ie it takes exact sequences of sheaves of \(\O\)-modules with \(\O\)-linear \(\sG\)-action to exact sequences of sheaves of \(\O\)-modules.
  \item  If \(\sheaf E\) is a sheaf of \(\O\)-modules with \emph{trivial} \(\sG\)-action, then \(\sS\times_{\sG}\sheaf E\cong \sheaf E\).
  \item Given arbitrary sheaves of \(\O\)-modules \(\sheaf E\) and \(\sheaf F\) with \(\O\)-linear \(\sG\)-actions, consider \(\sheaf E\oplus \sheaf F\), \(\sheaf E\tensor \sheaf F\), \(\Lambda^i\sheaf E\) and \(\sheaf E^\dual\) with the induced \(\sG\)-actions.  Then we have the following isomorphisms of \(\O\)-modules, natural in  \(\sheaf E\) and \(\sheaf F\):
    \begin{alignat*}{7}
      \sigma\colon&& \sS\times_{\sG}(\sheaf E\oplus \sheaf F) &\cong (\sS\times_{\sG}\sheaf E) \oplus (\sS\times_{\sG}\sheaf F)\\
      \theta\colon&& \sS\times_{\sG}(\sheaf E\tensor \sheaf F) &\cong (\sS\times_{\sG}\sheaf E) \tensor (\sS\times_{\sG}\sheaf F)\\
      \lambda\colon&& \sS\times_{\sG}(\Lambda^k\sheaf E) &\cong \Lambda^k(\sS\times_{\sG}\sheaf E)\\
      \eta\colon&& \sS\times_{\sG}(\sheaf E^\dual) &\cong (\sS\times_{\sG}\sheaf E)^\dual
    \end{alignat*}
  \end{compactenum}
\end{lem}

\begin{proof}
  \begin{asparaenum}[\it (i)]
  \item If we fix \(s\) and \(U\) as in the proof of \Cref{lem:local-iso},  then the induced isomorphism \(\sS\times_{\sG}\sheaf E_{|U}\to\sheaf E_{|U}\) is functorial for morphisms of \(\O\)-modules with \(\O\)-linear \(\sG\)-action.  The claim follows as exactness of a sequence of sheaves can be checked locally.

  \item When \(\sG\) acts trivially on \(\sheaf E\), the local isomorphisms of \Cref{lem:local-iso} do not depend on choices and glue to a global isomorphism.

  \item It is immediate from \Cref{lem:local-iso} that in each case the two sides are locally isomorphic, but we still need to construct global morphisms between them.

    For \(\oplus\) everything is clear.

    For \(\tensor\) and \(\Lambda^k\), we first note that all constructions involved are compatible with sheafification, in the following sense: let \(\pretensor\) and \(\preLambda\) denote the presheaf tensor product and the presheaf exterior power.  Then, for arbitrary presheaves \(\sheaf E\) and \(\sheaf F\),  the canonical morphisms
    \begin{align*}
      (\sheaf E\pretensor\sheaf F)^+ &\to (\sheaf E^+)\tensor(\sheaf F^+)\\
      (\preLambda^k\sheaf E)^+ &\to \Lambda^k(\sheaf E^+)\\
      (\sS\pretimes_{\sG}\sheaf E)^+ &\to \sS\times_{\sG}(\sheaf E^+)
    \end{align*}
    are isomorphisms.  (In the third case, this follows from \Cref{lem:local-iso}.)

    The arguments for \(\tensor\) and \(\Lambda^k\) are very similar, so we only discuss the latter functor.
    Let \(\prebigoplus\) denote the (infinite) direct sum in the category of presheaves.
    We first check that the morphism
    \begin{align*}
      \textstyle\prebigoplus_{\sS}(\preLambda^k\sheaf E) \to \preLambda^k(\textstyle\prebigoplus_{\sS}\sheaf E)
    \end{align*}
    which identifies the summand \(\p{s}{\preLambda^k\sheaf E}\) on the left with \(\preLambda^k(\p{s}{\sheaf E})\) on the right induces a well-defined morphism
    \begin{align*}
      \sS\pretimes_{\sG}(\preLambda^k\sheaf E) \xrightarrow{\hat\lambda} \preLambda^k(\sS\pretimes_{\sG}\sheaf E).
    \end{align*}
    Secondly, we claim that \(\hat\lambda\) is locally an isomorphism.  For this, we only need to observe that over any \(U\) such that \(\sS(U)\neq \emptyset\), we have a commutative triangle
    \[\xymatrix{
      {\left(\sS\pretimes_{\sG}(\preLambda^k\sheaf E)\right)(U)}\ar[dr]_{\cong} \ar[rr]^{\hat\lambda}&& {\left(\preLambda^k(\sS\pretimes_{\sG}\sheaf E)\right)(U)}\ar[dl]^{\cong}\\
      & {(\preLambda^k\sheaf E)(U)}
    }\]
    where the diagonal arrows are induced by the isomorphisms of \Cref{lem:local-iso}.

    For dualization, one of the sheafification morphisms goes in the wrong direction, so the argument is slightly different.  Again, we first construct a morphism of presheaves \(\hat\eta\colon\sS\pretimes_{\sG}\sheaf E^\dual \to (\sS\pretimes_{\sG}\sheaf E)^\dual\).
    Over opens \(U\) such that \(\sS(U)=\emptyset\), the left-hand side is zero, so we take the zero morphism.
    Over opens \(U\) with \(\sS(U)\neq\emptyset\), we define
    \begin{align*}
      (\sS\pretimes_{\sG}\sheaf E^\dual)(U) &\xrightarrow{\hat\eta} (\sS\pretimes_{\sG}\sheaf E)^\dual(U)\\
      \p{s}{\phi} \quad &\mapsto \left(\p{s.g}{v} \mapsto \phi(g.v)\right)
    \end{align*}
    Over these \(U\) with \(\sS(U)\neq\emptyset\), the morphism is in fact an isomorphism. To define a local inverse, pick an arbitrary \(s\in\sS(U)\), and send \(\psi\) on the right-hand side to \(\p{s}{v\mapsto \psi(\p{s}{v})}\) on the left.

    Finally, given the morphism \(\hat\eta\), we consider the following square in which \(\alpha\) and \(\beta\) are sheafification morphisms:
    \[\xymatrix{
      {\sS\pretimes_{\sG}\sheaf E^\dual}  \ar[r]^{\hat\eta}\ar[d]_{\alpha}\ar@{-->}[rd] & {(\sS\pretimes_{\sG}\sheaf E)^\dual} \\
      { \sS\times_{\sG}\sheaf E^\dual} \ar@{-->}[r]_{\eta} & {(\sS\times_{\sG}\sheaf E)^\dual} \ar[u]_{\beta^\dual}
    }\]
    By \Cref{lem:local-iso}, both \(\alpha\) and \(\beta\) are locally isomorphisms, and it follows that \(\beta^\dual\) is likewise locally an isomorphism.  The diagonal morphism is defined as follows:  over any \(U\) with \(\sS(U)=\emptyset\), it is the zero morphism, and over all other \(U\), it is the composition of \(\hat\eta\) with the (local) inverse of \(\beta^\vee\). Thus, the dotted diagonal is a factorization of \(\hat\eta\) over \((\sS\times_{\sG}\sheaf E)^\dual\).  The latter being a sheaf, this factorization must further factor through \(\sS\times_G\sheaf E^\dual\).  We thus obtain the horizontal morphism of sheaves \(\eta\) which, being a locally an isomorphism, must be an isomorphism.\qedhere
  \end{asparaenum}
\end{proof}

\subsubsection*{Twisting symmetric bundles}
Recall that a duality functor is a functor between categories with dualities
\[F\colon (\cat A, \vee, \omega) \to (\cat B, \vee, \omega)\]
together with a natural isomorphism
\(\eta\colon F(-^\vee) \xrightarrow{\;\cong\;} F(-)^\vee\)
such that
\[\xymatrix{
  & {FA} \ar[dl]_{F\omega_A} \ar[dr]^{\omega_{FA}} & \\
  {F({A^\vee}^\vee)} \ar[r]_{\eta_{A^\vee}} & {F(A^\vee)^\vee} & \ar[l]^{\eta_A^\vee} {{(FA)^\vee}^\vee}
}\]
commutes \cite{Balmer:Handbook}*{Def.~1.1.15}.
Such a functor induces a functor \(F_\sym\) on the category of symmetric spaces over \(\cat A\): it sends a symmetric space \((A,\alpha)\) over \(\cat A\) to the symmetric space \((FA,\eta_AF\alpha)\) over \(\cat B\). Moreover, we have isometries \(H(FA)\cong F_\sym(HA)\), where \(H\) denotes the hyperbolic functor.  We will sometimes simply write \(F\) for \(F_\sym\).

Now consider the functor \(\sS\times_{\sG}-\).  By \Cref{lem:local-iso}, we can restrict \(\sS\times_{\sG}-\) to a functor from \(\sG\)-equivariant vector bundles to vector bundles over \(X\).  Moreover, by \Cref{lem:basic-properties}, we have a natural isomorphism \(\eta\colon \sS\times_{\sG}(-^\dual)\cong(\sS\times_{\sG}(-))^\dual\).   The commutativity of the triangle in the definition of a duality functor is easily checked, so we deduce:
\begin{lem}
  \((\sS\times_{\sG}-, \eta)\) is a duality functor \(\sG\cat Vec(X) \to  \cat Vec(X)\).
\end{lem}
For any symmetric vector bundle \((\sheaf E,\eps)\) with an \(\O\)-linear left \(\sG\)-action, we can now define
\[
\sS\times_{\sG}(\sheaf E,\eps):=(\sS\times_{\sG}-)_\sym (\sheaf E,\eps).
\]
No compatibility of the \(\sG\)-action with the symmetric structure is required for this definition, but we will insist in the following that \(\sG\) acts via isometries:
\begin{lem}\label{lem:local-isometry}
  If \(\sG\) acts on \(\sheaf E\) via isometries, then \(\sS\times_{\sG}(\sheaf E,\eps)\) is locally \emph{isometric} to \((\sheaf E,\eps)\).
\end{lem}
More precisely, the local isomorphisms of \Cref{lem:local-iso} are isometries in this case.
Moreover, the natural isomorphisms \(\sigma\), \(\theta\) and \(\lambda\) of Lemma~\ref{lem:basic-properties} respect the symmetric structures:
\begin{lem}\label{lem:basic-isometries}
 For symmetric vector bundles \((\sheaf E,\eps)\) and \((\sheaf F,\phi)\) on which \(\sG\) acts through isometries, \(\sigma\), \(\theta\) and \(\lambda\) respect the induced symmetries.
\end{lem}
\begin{proof}
  Temporarily writing \(F\) for the functor \(\sS\times_{\sG}-\),  checking the claim for \(\theta\) amounts to the following:
  For symmetric bundles \((\sheaf E,\eps)\) and \((\sheaf F,\phi)\), the isomorphism \(\theta_{\sheaf E,\sheaf F}\) is an isometry from \(F(\sheaf E\tensor \sheaf F)\) to \(F\sheaf E\tensor F\sheaf F\) with respect to the induced symmetries if and only if the outer square of the following diagram commutes:
  \[\xymatrix{
    {F(\sheaf E\tensor \sheaf F)} \ar[d]^{F(\eps\tensor\phi)} \ar[rr]^{\theta_{\sheaf E,\sheaf F}}
    &&{F\sheaf E\tensor F\sheaf F} \ar[d]^{F\eps\tensor F\phi} \\
    {F(\sheaf E^\vee\tensor \sheaf F^\vee)} \ar[d]  \ar[rr]^{\theta_{\sheaf E^\vee,\sheaf F^\vee}}
    &&{F(\sheaf E^\vee)\tensor F(\sheaf F^\vee)} \ar[d]^{\eta_\sheaf E\tensor\eta_\sheaf F} \\
    {F((\sheaf E\tensor \sheaf F)^\vee)}  \ar[d]^{\eta_{\sheaf E\tensor \sheaf F}}
    && {(F\sheaf E)^\vee\tensor (F\sheaf F)^\vee} \ar[d]     \\
    {(F(\sheaf E\tensor \sheaf F))^\vee}
    &&\ar[ll]^{\theta^\vee_{\sheaf E,\sheaf F}} {(F\sheaf E\tensor F\sheaf F)^\vee}
  }\]

  Similarly, for a symmetric vector bundle \((\sheaf E,\eps)\), the isomorphism \(\lambda_{\sheaf E}\) is an isometry from \(F(\Lambda^k\sheaf E)\) to \(\Lambda^k(F\sheaf E)\) if and only if the outer square of the following diagram commutes:
  \[\xymatrix{
    {F\Lambda^k\sheaf E} \ar[d]^{F\Lambda^k\eps} \ar[rr]^{\lambda_{\sheaf E}}
    && {\Lambda^k F\sheaf E} \ar[d]^{\Lambda^k F\eps}  \\
    {F\Lambda^k(\sheaf E^\vee)} \ar[d] \ar[rr]^{\lambda_{\sheaf E^\vee}}
    && {\Lambda^kF(\sheaf E^\vee)} \ar[d]^{\Lambda^k\eta_\sheaf E} \\
    {F((\Lambda^k\sheaf E)^\vee)}  \ar[d]^{\eta_{\Lambda^k\sheaf E}}
    && {\Lambda^k((F\sheaf E)^\vee)} \ar[d] \\
    {(F\Lambda^k\sheaf E)^\vee}
    &&\ar[ll]^{\lambda^\vee_{\sheaf E}} {(\Lambda^k F\sheaf E)^\vee}
  }\]

  In both cases, we already know that the the upper square commutes for all \(\sheaf E\) and \(\sheaf F\), by naturality of \(\theta\) and \(\lambda\). So it suffices to verify that the lower  square commutes. This can be checked locally, and follows easily from the descriptions of \(\eta\), \(\theta\) and \(\lambda\) given in the proof of \Cref{lem:local-iso}.
\end{proof}

\subsubsection*{Proofs of the statements}
\Cref{lem:basic-properties,lem:basic-isometries} immediately imply:
\begin{cor}\label{lem:K2GW2} Let \(\pi\colon X\to \Spec(k)\) be a scheme over some field \(k\).  For any algebraic group scheme \(G\) over \(k\), and for any \(G\)-torsor \(\sS\), the maps
  \begin{align*}
    \K(\Rep_kG) &\to \K(X)    &  \GW(\Rep_kG)&\to \GW(X)\\
    V &\mapsto \sS\times_G \pi^*V   &      (V,\nu) &\mapsto \sS\times_{G}\pi^*(V,\nu)
  \end{align*}
  are \(\lambda\)-morphisms.
\end{cor}
This proves statements (K2) and (GW2).
It remains to prove (K3) and (GW3), which we now state in a more detailed form.
\begin{prop}\label{lem:K3}
  Let \(V_n\) be the standard representation of \(\GL{n}\).
  \begin{compactenum}[(a)]
  \item Any vector bundle \(\sheaf E\) is isomorphic to \(\sS\times_{\GL{n}}\pi^*V_n\) for some Zariski \(\GL{n}\)-torsor \(\sS\).
  \item For any two vector bundles \(\sheaf E\) and \(\sheaf F\), there exists a Zariski \(\GL{n}\times\GL{m}\)-torsor \(\sS\) such that
    \begin{align*}
      \sheaf E &\cong \sS\times_{\GL{n}\times\GL{m}} \pi^* V_n\\
      \sheaf F&\cong \sS\times_{\GL{n}\times\GL{m}} \pi^* V_m
    \end{align*}
    Here \(\GL{m}\) is supposed to act trivially on \(V_n\),\\
    and \(\GL{n}\) is supposed to act trivially on \(V_m\).
  \end{compactenum}
\end{prop}
\begin{prop}\label{lem:GW3}
  Let \(X\) be a scheme over a field \(k\) of characteristic not two.
  Let \((V_n,q_n)\) be the standard representation of \(\OO{n}\) over \(k\), equipped with its standard symmetric form.
  \begin{compactenum}[(a)]
  \item Any symmetric vector bundle \((\sheaf E,\eps)\) is isomorphic to \(\sS\times_{\OO{n}}\pi^*(V_n,q_n)\) for some ^^e9tale \(\OO{n}\)-torsor \(\sS\).
  \item For any two symmetric vector bundles \((\sheaf E,\eps)\) and \((\sheaf F,\phi)\), there exists an ^^e9tale \(\OO{n}\times\OO{m}\)-torsor \(\sS\) such that
    \begin{align*}
      (\sheaf E,\eps)  &\cong \sS\times_{\OO{n}\times\OO{m}} \pi^*(V_n,q_n)\\
      (\sheaf F,\phi) &\cong \sS\times_{\OO{n}\times\OO{m}} \pi^*(V_m,q_m)
    \end{align*}
    Here, \(\OO{m}\) is supposed to act trivially on \(V_n\),\\
    and \(\OO{n}\) is supposed to act trivially on \(V_m\).
  \end{compactenum}
\end{prop}
\begin{proof}[Proof of \Cref{lem:K3}]\mbox{}\hfill
  \begin{asparaenum}[(a)]
  \item  Identify \(\pi^*V_n\) with \(\O^{\oplus n}\), and let \(\sS\) be the sheaf of isomorphisms \(\sS:=\sheaf Iso(\O^{\oplus n},\sheaf E)\) with \(\GL{n} = \sheaf Aut(\O^{\oplus n})\) acting by precomposition.
    This is a \(\GL{n}\)-torsor as \(\sheaf E\) is locally isomorphic to \(\O^{\oplus n}\).
    Moreover, we have a well-defined morphism
    \begin{align*}
      \ev\colon \sS\pretimes_{\sG}\O^{\oplus n}&\to \sheaf E\\
      \p{f}{v}&\mapsto f(v)
    \end{align*}
    which is locally an isomorphism:
    for any \(s\in\sheaf Iso(\O^{\oplus n}, \sheaf E)(V)\), the restriction \(\ev_{|V}\) factors as
    \[
    (\sS\pretimes_{\sG}\O^{\oplus n})_{|V} \xrightarrow{\cong} \O^{\oplus n}_{|V}  \xrightarrow[s]{\cong} \sheaf E_{|V},
    \]
    where the first arrow is the isomorphism \(\p{f}{v}\mapsto g_f(v)\) of \Cref{lem:local-iso} determined by \(s\).
  \item Suppose \(\sS\) is a \(\sG\)-torsor, \(\sS'\) is a \(\sG'\)-torsor, and \(\sheaf E\) is a sheaf of \(\O\)-modules with \(\O\)-linear actions by both \(\sG\) and  \(\sG'\).  Then if \(\sG'\) acts trivially,
    \[ (\sS\times\sS')\times_{\sG\times\sG'}\sheaf E \cong \sS\times_{\sG}\sheaf E.\]
    We can therefore take \(\sS:=\sheaf Iso(\O^{\oplus n},\sheaf E)\times\sheaf Iso(\O^{\oplus m},\sheaf F)\).
    \qedhere
  \end{asparaenum}
\end{proof}
\begin{proof}[Proof of \Cref{lem:GW3}]
  We have \(\OO{n} = \sheaf Aut(\O^{\oplus n},q_n)\).
  So let \(\sS\) be the sheaf of isometries \(\sS:=\sheaf Iso((\O^{\oplus n},q_n),(\sheaf E,\eps))\) with \(\OO{n}\) acting by precomposition.  This is an \(\OO{n}\)-torsor in the ^^e9tale topology since any symmetric vector bundle \((\sheaf E,\eps)\) is ^^e9tale locally isometric to \((\O^{\oplus n},q_n)\) (\eg \cite{Hornbostel:Representability}*{3.6}).
  The rest of the proof works exactly as in the non-symmetric case.
\end{proof}

\section{(No) splitting principle}\label{sec:no splitting}
The splitting principle in K-theory asserts that any vector bundle behaves like a sum of line bundles.  There are two incarnations:

{\bf The algebraic splitting principle: } For any positive element \(e\) of a \(\lambda\)-ring with positive structure \(A\), there exists an extension of \(\lambda\)-rings with positive structure \(A\hookrightarrow A_e\) such that \(e\) splits as a sum of line elements in \(A_e\).

{\bf The geometric splitting principle: } For any vector bundle \(\sheaf E\) over a scheme \(X\), there exists an \(X\)-scheme \(\pi\colon X_{\sheaf E}\to X\) such that the induced morphism \(\pi^*\colon \K(X)\hookrightarrow \K(X_{\sheaf E})\) is an extension of \(\lambda\)-rings with positive structure, and such that \(\pi^*\sheaf E\) splits as a sum of line bundles in \(\K(X_{\sheaf E})\).

Both incarnations are discussed in \cite{FultonLang}*{I, \S2}.  An {\bf extension} of a \(\lambda\)-ring with positive structure \(A\) is simply an injective \(\lambda\)-morphism to another \(\lambda\)-ring with positive structure  \(A\hookrightarrow A'\), compatible with the augmentation and such that \(A_{\geq 0}\) maps to \(A'_{\geq 0}\).

\subsubsection*{No splitting principle for GW}
For \(\GW(X)\), the analogue of the geometric splitting principle fails:
\begin{quote}
  Over any field of characteristic not two, there exists a (smooth, projective) scheme \(X\) and a symmetric vector bundle \((\sheaf E,\eps)\) over \(X\) such that there exists no \(X\)-scheme \(\pi\colon X_{(\sheaf E,\eps)}\to X\) for which the class of \(\pi^*(\sheaf E,\eps)\) in \(\GW(X_{(\sheaf E,\eps)})\) splits into a sum of symmetric line bundles.
\end{quote}
The natural analogue of the algebraic splitting principle could be formulated using the notion of a real \(\lambda\)-ring:

\begin{defn}
A \define{real \(\lambda\)-ring} is a \(\lambda\)-ring with positive structure \(A\) in which any line element squares to one.
\end{defn}
 This property is clearly satisfied by the \GrothendieckWitt ring \(\GW(X)\) of any scheme \(X\).
However, an algebraic splitting principle for real \(\lambda\)-rings fails likewise:
\begin{quote}
  There exist a real \(\lambda\)-ring \(A\) and a positive element \(e\in A\)  that does not split into a sum of line elements in any extension of real \(\lambda\)-rings \(A\hookrightarrow A_e\).
\end{quote}
The failure of both splitting principles is clear from the following simple counterexample:
\begin{lem}  Let \(\PP^2\) be the projective plane over some field \(k\) of characteristic not two.  Consider the element \(e:=H(\O(1))\in\GW(\PP^2)\).  There exists no extension of \(\lambda\)-rings \(\GW(\PP^2)\hookrightarrow A_e\) such that \(A_e\) is real and such that \(e\) splits as a sum of line elements in \(A_e\).
\end{lem}
\begin{proof}
  For any element \(a\) in a real \(\lambda\)-ring that can be written as a sum of line elements, the Adams operations \(\psi^n\) are given by
  \[
  \psi^n(a) = \begin{cases}
    \rank(a) &\text{if \(n\) is even,}\\
    a        &\text{if \(n\) is odd.}
  \end{cases}
  \]
  However, for \(e:= H(\O(1))\in\GW(\PP^2)\) we have \(\psi^2(e) \neq 2\),
  so \(\psi^2(e)\) cannot be a sum of line bundles, neither in \(\GW(\PP^2)\) itself nor in any real extension.

  (Explicitly, \(\GW(\PP^2) = \pi^*\GW(k) \oplus \ZZ e\) with \(e^2=-2\pi^*\langle 1, -1 \rangle + 4e\)  (see \Cref{eg:Pr} below) and \(\lambda^2(e) = \pi^*\langle -1 \rangle\) (by an explicit calculation).  So
  \[
  \psi^2(e)
  = e^2 -2\lambda^2(e)
  %= 4e -2\pi^*\langle 1, -1 \rangle -2\pi^*\langle -1 \rangle
  = -2\pi^*\langle 1, -1, -1 \rangle + 4e
  ,
  \]
  which differs from \(2\), as claimed.)
 \end{proof}

\subsubsection*{A splitting principle for ^^e9tale cohomology}
Despite the negative result above, we do have a splitting principle for Stiefel-Whitney classes of symmetric bundles.  Let \(X\) be any scheme over \(\ZZ[\frac{1}{2}]\).
\begin{prop}\label{thm:etale-splitting}
  For any symmetric bundle \((\vb E, \eps)\) over \(X\) there exists a morphism
  \(
  \pi\colon X_{(\vb E, \eps)} \to X
  \)
  such that \(\pi^*(\vb E, \eps)\) splits as an orthogonal sum of symmetric line bundles over \(X_{(\vb E,\eps)}\) and such that \(\pi^*\) is injective on ^^e9tale cohomology with \(\ZZ/2\)-coefficients.
\end{prop}
\begin{proof}
  Recall the geometric construction of higher Stiefel-Whitney classes of Delzant and Laborde, as explained for example in \cite{EKV}*{\S5}:  given a symmetric vector bundle \((\sheaf E,\eps)\) as above, the key idea is to consider the scheme of non-degenerate one-dimensional subspaces \(\pi\colon \PP_{\nondeg}(\sheaf E,\eps) \to X\), \ie the complement of the quadric in \(\PP(\sheaf E)\) defined by \(\eps\).  (This is an algebraic version of the projective bundle associated with a real vector bundle in topology; \cf \cite{Me:WCCV}*{Lem.~1.7}.) Let \(\O(-1)\) denote the restriction of the universal line bundle over \(\PP(\sheaf E)\) to \(\PP_{\nondeg}(\sheaf E,\eps)\).  This is a subbundle of \(\pi^*\sheaf E\), and by construction the restriction of \(\pi^*\eps\) to \(\O(-1)\) is non-degenerate.  Let \(w\) be the first Stiefel-Whitney class of this symmetric line bundle \(\O(-1)\). The ^^e9tale cohomology of \(\PP_{\nondeg}(\sheaf E, \eps)\) decomposes as
\[
   H^*_{\et}(\PP_{\nondeg}(\sheaf E,\eps),\ZZ/2)=\bigoplus_{i=0}^{r-1} \pi^*H^*_{\et}(X,\ZZ/2)\cdot w^i,
\]
and the higher Stiefel-Whitney classes of \((\sheaf E,\eps)\) can be defined as the coefficients of the equation expressing \(w^r\) as a linear combination of the smaller powers \(w^i\) in \(H^*_{\et}(\PP_{\nondeg}(\sheaf E,\eps),\ZZ/2)\).

We only need to note two facts from this construction:  Firstly, over \(\PP_{\nondeg}(\sheaf E,\eps)\) we have an orthogonal decomposition
\[
  \pi^*(\sheaf E,\eps) \cong (\O(-1),\epsilon') \perp (\sheaf E'',\eps''),
\]
where \(\sheaf E'' = \O(-1)^\perp\) and \(\eps'\) and \(\eps''\) are the restrictions of \(\pi^*\eps\).   Secondly, \(\pi\)  induces a monomorphism from the ^^e9tale cohomology of \(X\) to the ^^e9tale cohomology of \(\PP_{\nondeg}(\sheaf E,\eps)\).
So the \namecref{thm:etale-splitting} is proved by iterating this construction.
\end{proof}

\section{The $\gamma$-filtration on the Grothendieck-Witt ring}
From now on, we assume that \(X\) is connected.
As we have seen, \(\GW(X)\) is a pre-\(\lambda\)-ring with positive structure, and we can consider the associated \(\gamma\)-filtration \(\gammaF{i} \GW(X)\) of \(\GW(X)\).  The image of this filtration under the canonical epimorphism \(\GW(X)\twoheadrightarrow \W(X)\) will be denoted \(\gammaF{i} \W(X)\).  In particular, by definition,
\begin{alignat*}{7}
&\gammaF{1} \GW(X) &&= \clasF{1}\GW(X) &&:= \ker(\rank\colon \GW(X)\to \ZZ),\\
&\gammaF{1} \W(X) &&= \clasF{1}\W(X) &&:= \ker(\bar\rank\colon \W(X)\to \ZZ/2).
\end{alignat*}

For a field, or more generally for a connected semi-local ring \(R\), we also write \(\GI(R)\) and \(\I(R)\) instead of \(\clasF{1}\GW(R)\) and \(\clasF{1}\W(R)\), respectively.

\subsection{Comparison with the fundamental filtration}
\begin{prop}\label{prop:local-filtration}
  For any connected semi-local commutative ring \(R\) in which two is invertible, the \(\gamma\)-filtration on \(\GW(R)\) is the filtration by powers of the augmentation ideal \(\GI(R)\), and the induced filtration on \(\W(R)\) is the filtration by powers of the fundamental ideal \(I(R)\).
\end{prop}
\begin{proof}
  As we have already noted in the proof of \Cref{GW-R-special}, all positive elements of the \GrothendieckWitt ring \(\GW(R)\) can be written as sums of line elements. Thus, the claim concerning \(\GW(R)\) is immediate from \Cref{lem:gamma-filtration-generators}.  Moreover, the fundamental filtration on \(\W(R)\) is the image of the fundamental filtration on \(\GW(R)\).
\end{proof}

\begin{rem}
  In the situation above, the projection \(\GW(R)\to\W(R)\) even induces isomorphisms \( \GI^n(R) \to \I^n(R)\), so that \(\gr^i_\gamma\GW(R) \cong \gr^i_\gamma\W(R)\) in degrees \(i>0\).  This fails for general schemes in place of \(R\) (see \Cref{sec:examples}).
\end{rem}

\begin{rem}
  It may seem more natural to define a filtration on \(\GW(X)\) starting with the kernel not of the rank morphism but of the rank reduced modulo two, as for example in \cite{Auel:Milnor}:
  \[
  \GI'(X) := \ker\big(\GW(X)\to H^0(X,\ZZ/2)\big)
  \]
  For connected \(X\), \(\GI'(X)\) is isomorphic to a direct sum of \(\GI(X)\) and a copy of \(\ZZ\) generated by the hyperbolic plane \(\HH\).
  In particular, \(\GI(X)\) and \(\GI'(X)\) have the same image in \(\W(X)\).
  However, even over a field, the filtration by powers of \(\GI'\) does \emph{not} yield the same graded ring as the filtration by powers of (\(\GI\) or) \(\I\).
  For example, for \(X=\Spec(\RR)\), we find:
 \begin{align*}
    \factor{\GI^n(\RR)}{\GI^{n+1}(\RR)} &\cong \ZZ/2 &&\quad(n>0)\\
    \factor{(\GI')^n(\RR)}{(\GI')^{n+1}(\RR)} &\cong \ZZ/2 \oplus \ZZ/2
  \end{align*}
  It is the filtration by powers of \(\GI\) that yields an associated graded ring isomorphic to \(H^*_{\et}(\RR, \ZZ/2)\) in positive degrees, not the filtration by powers of \(\GI'\).
\end{rem}

\subsection{Comparison with the classical filtration}
A common filtration on the Witt ring of a scheme is given by the kernels of the first two ^^e9tale Stiefel-Whitney classes \(w_1\) and \(w_2\) on the \GrothendieckWitt ring and of the induced classes \(\bar w_1\) and \(\bar w_2\) on the Witt ring:
\begin{alignat*}{7}
  &\clasF{2}\GW(X) &&:= \ker\left(\clasF{1}\GW(X) \xrightarrow{w_1} H^1_\et(X,\ZZ/2)\right) \\
  &\clasF{2}\W(X) &&:= \ker\left(\clasF{1}\W(X)\phantom{G}\xrightarrow{\bar w_1} H^1_\et(X,\ZZ/2)\right) \\\\
  &\clasF{3}\GW(X) &&:= \ker\left(\clasF{2}\GW(X)  \xrightarrow{w_2} H^2_\et(X,\ZZ/2)\right) \\
  &\clasF{3}\W(X) &&:= \ker\left(\clasF{2}\W(X)\phantom{G}\xrightarrow{\bar w_2}  H^2_\et(X,\ZZ/2)/\Pic(X)\right)
\end{alignat*}
\begin{prop}\label{comparison:F2}
  Let \(X\) be any connected scheme over a field of characteristic not two (or, more generally, any scheme such that the canonical pre-\(\lambda\)-structure on \(\GW(X)\) is a
 \(\lambda\)-structure). Then:
  \begin{align*}
    \gammaF{2} \GW(X) &= \clasF{2}\GW(X) \\
    \gammaF{2} \W(X) &= \clasF{2}\W(X)
  \end{align*}
\end{prop}
\begin{proof}
  The first identity is a consequence of \Cref{lem:graded-degree-1}:
  In our case, the group of line elements may be identified with \(H^1_\et(X,\ZZ/2)\); then the determinant
  \(
  \GW(X) \to H^1_\et(X,\ZZ/2)
  \)
  is precisely the first Stiefel-Whitney class \(w_1\). In particular, the kernel of the restriction of \(w_1\) to \(\clasF{1}\GW(X)\) is \(\gammaF{2} \GW(X)\), as claimed.
  For the second identity, it suffices to observe that \(\clasF{2}\GW(X)\) maps surjectively onto \(\clasF{2}\W(X)\).
\end{proof}

In order to analyse the relation of \(\gammaF{3} \GW(X)\) to \(\clasF{3}\GW(X)\), we need a few lemmas concerning products of ``reduced line elements'':
\begin{lem}\label{lem:gamma-calculation}
  Let \(u_1,\dots, u_l, v_1,\dots, v_l\) be line elements in a pre-\(\lambda\)-ring \(A\) with positive structure. Then
  \(\gamma^{k}\left(\textstyle\sum_i(u_i-v_i)\right) \)
  can be written as a linear combination of products
  \[
  (u_{i_1}-1)\cdots(u_{i_s}-1)(v_{j_1}-1)\cdots(v_{j_t}-1)
  \] with \(s+t=k\) factors.
\end{lem}

\begin{proof}
  This is easily seen by induction over \(l\).
 For \(l=1\) and \(k=0\) the statement is trivial, while for \(l=1\) and \(k\geq 1\) we have
  \begin{align*}
    \gamma^{k}(u-v) &=\gamma^{k}((u-1)+(1-v)) \\
    &=\gamma^{0}(u-1)\gamma^{k}(1-v) + \gamma^{1}(u-1)\gamma^{k-1}(1-v)\\
    &=\pm(v-1)^k \mp (u-1)(v-1)^{k-1}
  \end{align*}
  For the induction step, we observe that every summand in
  \begin{align*}
    \gamma^{k}\left(\textstyle\sum_{i=1}^{l+1}u_i-v_i\right) &= \sum_{i=0}^k\gamma^{i}\left(\textstyle\sum_{i=1}^l u_i-v_i\right)\gamma^{k-i}(u_l-v_l)
  \end{align*}
  can be written as a linear combination of the required form.
\end{proof}

\begin{lem}\label{lem:w-calculation}
  Let \(X\) be a scheme over \(\ZZ[\frac{1}{2}]\), and let \(u_1, \dots, u_n \in \GW(X)\) be classes of symmetric line bundles with Stiefel-Whitney classes  \(w_1(u_i)=:\bar u_i\). Let \(\rho\) denote the product
  \[
  \rho := (u_1-1)\cdots(u_n-1).
  \]
  Then \(w_i(\rho) = 0\) for \(0 < i < 2^{n-1}\), and
  \[
  w_{2^{n-1}}(\rho)
  = \quad\prod_{\mathclap{\substack{1\leq i_1 < \dots < i_k \leq n\\\text{ with } k\text{ odd}}}}\quad ({\bar u}_{i_1} + \cdots + {\bar u}_{i_k})
  = \quad\sum_{\mathclap{\substack{r_1,\dots, r_n:\\2^{r_1}+\cdots+2^{r_n}= 2^{n-1}}}}\quad {\bar u}_1^{2^{r_1}}\cdots\cdot {\bar u}_n^{2^{r_n}}.
  \]
\end{lem}

\begin{proof}
  The lemma generalizes Lemma~3.2\slash Corollary~3.3 of \cite{Milnor}.  The first part of Milnor's proof applies verbatim.
  Consider the evaluation map
  \[
  \ZZ/2\llbracket x_1,\dots,x_n\rrbracket \xrightarrow{\;\ev\;}  \textstyle\prod_i H^i_\et(X,\ZZ/2)\\
  \]
  sending \(x_i\) to \(\bar u_i\).
  The total Stiefel-Whitney class \(w(\rho) = 1 + w_1(\rho) + w_2(\rho) + \dots\) is the  evaluation of the power series
  \[
  \omega(x_1,\dots,x_n) := \left(\frac{\prod_{\abs{\vec \epsilon}\text{ even }}(1+\vec\epsilon\,\vec x)}
    {\prod_{\abs{\vec \epsilon}\text{ odd }}(1+\vec\epsilon\,\vec x)}\right)^{(-1)^n},
  \]
  where the products range over all \(\vec \epsilon = (\epsilon_1,\dots,\epsilon_n)\in(\ZZ/2)^n\) with \(\abs{\vec \epsilon}:=\epsilon_1+\cdots+\epsilon_n\) even or odd, and where \(\vec\epsilon\,\vec x\) denotes the sum \(\sum_i\epsilon_i x_i\).
  As Milnor points out, all factors of \(\omega\) cancel if we substitute \(x_i=0\) for some~\(i\).
  More generally, all factors cancel whenever we replace a given variable \(x_i\) by the sum of an even number of variables \(x_{i_1} + \cdots + x_{i_{2l}}\) all distinct from \(x_i\).
  Indeed, consider the substitution \(x_n = \vec\alpha\,\vec x\) with \(\abs{\vec\alpha}\) even and \(\alpha_n= 0\). Write \(\vec x = (\vec x', x_n)\), \(\vec \epsilon= (\vec\epsilon',\epsilon_n)\) and \(\vec \alpha=(\vec \alpha',0)\), so that the substitution may be rewritten as \(x_n = \vec \alpha'\,\vec x'\). Then
  \[
  (\vec \epsilon', \epsilon_n)(\vec x', \vec\alpha'\,\vec x') = (\vec \epsilon'+\vec \alpha', \epsilon_n+1)(\vec x', \vec\alpha'\,\vec x'),
  \]
  but the parities of \(\abs{(\vec\epsilon',\epsilon_n)}\) and \(\abs{(\vec \epsilon'+\vec \alpha',\epsilon_n+1)}\) are different. Thus, the corresponding factors of \(\omega\) cancel.
  It follows that \(\omega-1\) is divisible by all sums of an odd number of distinct variables \(x_{i_1} + \cdots + x_{i_k}\).
  Therefore,
  \begin{equation}\label{eq:w-calculation}
  \omega = 1 + (\textstyle\prod_{\abs{\vec\epsilon}\text{ odd}} \vec\epsilon \vec x)\cdot f(\vec x)
  \end{equation}
  for some power series \(f\).
  In particular, \(\omega\) has no non-zero coefficients in positive total degrees below \(\sum_{k \text{ odd}}\binom{n}{k} = 2^{n-1}\), proving the first part of the \namecref{lem:w-calculation}.

  For the second part, we need to show that the constant coefficient of \(f\) is \(1\).
  This can be seen by considering the substitution \(x_1=x_2=\cdots=x_n = x\) in \eqref{eq:w-calculation}:
  we obtain
  \begin{align*}
  \left(\frac{1}{(1+x)^K}\right)^{\pm 1} &= 1 + x^K f(x,\dots,x)
  \intertext{with \(K = \sum_{k \text{ odd}} \binom{n}{k} = 2^{n-1}\),
  and as \((1+x^K) = 1 + x^K \mod 2\) for \(K\) a power of two,
  this equation can be rewritten as
  }
  (1 + x^K)^{\mp 1} &= 1 + x^K f(x,\dots, x).
  \end{align*}
  The claim follows.
  Finally, the identification of the product expression for \(w_{2^{n-1}}(\rho)\) with a sum is Lemma~2.5 of \cite{GuillotMinac}.
  It is verified by showing that all factors of the product divide the sum, using similar substitution arguments as above.
\end{proof}

\begin{rem}
  Milnor's proof in the case when \(X\) is a field \(k\) uses the relation \(a^{\scup 2} = [-1]\scup a\) in \(H^2(k,\ZZ/2)\), which does not hold in general.
\end{rem}

\begin{prop}\label{comparison:F2F3}
  Let \(X\) be a connected scheme over \(\ZZ[\frac{1}{2}]\).
  Then
  \(
  w_i(\gammaF{n} \GW(X)) = 0 \text{ for } 0<i<2^{n-1}
  \).
  In particular:
  \begin{align*}
  \gammaF{2}\GW(X)&\subset \clasF{2}\GW(X)\\
  \gammaF{3}\GW(X)&\subset \clasF{3}\GW(X)
  \end{align*}
\end{prop}
\begin{proof}%[Proof of the proposition]
  Let \(x:= \gamma^{k_1}(x_1)\cdots\gamma^{k_l}(x_l)\) be an additive generator of \(\GW^n(X)\), \ie \(x_i\in \ker(\rank)\) and \(\sum k_i\geq n\). By writing each \(x_i\) as \([\vb E_i,\epsilon_i]-[\vb F_i,\phi_i]\) for certain symmetric vector bundles \((\vb E_i,\epsilon_i)\) and \((\vb F_i,\phi_i)\) and successively applying the splitting principal for ^^e9tale cohomology (\Cref{thm:etale-splitting}) to each of these, we can find a morphism
  \[ X_{x} \to X\]
  which is injective on ^^e9tale cohomology with \(\ZZ/2\)-coefficients, and such that each \(\pi^*x_i\) is a sum of differences of line bundles. By \Cref{lem:gamma-calculation}, each \(\gamma^{k_i}(\pi^*x_i)\) can therefore be written as a linear combination of products \((u_1-1)\cdots(u_m-1)\) with \(m=k_i\) factors, where each \(u_i\) is the class of some line bundle over \(X_{x}\). Using the naturality of the \(\gamma\)-operations, it follows that \(\pi^*x\) can be written as a linear combination of such products with \(m\geq n\) factors.

  By \Cref{lem:w-calculation}, the classes \(w_i\) vanish on every summand of this linear combination for \(0<i< 2^{n-1}\). So \(w_i(\pi^*x) = 0\) for all \(0<i<2^{n-1}\), and by the naturality of Stiefel-Whitney classes and the injectivity of \(\pi^*\) on cohomology we may conclude that \(w_i(x)\) vanishes in this range.
\end{proof}

\subsection{Comparison with the unramified filtration}\label{sec:unramified-filtration}
Here, we quickly summarize some observations on the relation of the \(\gamma\)-filtration with the ``unramified filtration''.

First, let \(X\) be an integral scheme with function field \(K\), and let \(\urF{*}\GW(X)\) denote the unramified filtration of \(\GW(X)\), given by the preimages of \(\GI^i(K)\) under the natural map \(\GW(X)\to \GW(K)\).  Said map is a morphism of augmented \(\lambda\)-rings, so \(\gammaF{i}\GW(X)\) maps to \(\gammaF{i}\GW(K)=\GI^i(K)\) and we obtain:

\begin{prop}\label{comparison:finer-than-unramified}
  For any integral scheme \(X\), the \(\gamma\)-filtration on \(\GW(X)\) is finer than the unramified filtration, \ie \(\gammaF{i}\GW(X)\subset \urF{i}\GW(X)\) for all~\(i\).\qed
\end{prop}

The unramified \GrothendieckWitt group of \(X\) is defined as
\[
\GW_\ur(X) := \bigcap_{x\in X^{(1)}} \im\big( \GW(\O_{X,x}) \to \GW(K)\big),
\]
where \(X^{(1)}\) denotes the set of codimension one points of~\(X\).
Let us consider the functors \(\GW\) and \(\GW_\ur\) as the presheaves on our given integral scheme \(X\) that send an open subset \(U\subset X\) to \(\GW(U)\) or \(\GW_\ur(U)\), respectively.  Then \(\GW_\ur\) is a sheaf, and we have a sequence of morphisms of presheaves
\[
\GW \to \GW^+ \to \GW_\ur \hookrightarrow \GW(K),
\]
where \((-)^+\) denotes sheafification and \(\GW(K)\) is to be interpreted as the constant sheaf with value \(\GW(K)\).  The unramified filtration of \(\GW_\ur\) is obtained by intersecting the fundamental filtration on \(\GW(K)\) with \(\GW_\ur\):
\[
\urF{i}\GW_\ur := \GW_\ur \cap \GI^i(K).
\]
This is a filtration by sheaves, and the unramified filtration \(\urF{i}\GW\) is given by the preimage of \(\urF{i}\GW_\ur\) under the above morphisms.

When \(X\) is regular integral of finite type over a field of characteristic not two, the purity results of Ojanguren and Panin \citelist{\cite{Ojanguren}\cite{OjangurenPanin}*{Thm~A}} imply that the morphism \(\GW^+\to \GW_\ur\) is an isomorphism.
If we further assume that the field is infinite, a result of Kerz and M^^fcller-Stach yields the following:

\begin{prop}\label{prop:sheafified-filtrations}
  For any regular integral scheme of finite type over an infinite field of characteristic not two, the  \(\gamma\)-filtration and the unramified filtration have the same sheafifications:
  \[
  (\gammaF{i}\GW)^+ = (\urF{i}\GW)^+ = \urF{i}\GW_\ur
  \]
\end{prop}
\begin{proof}
  As already mentioned, the results of Ojanguren and Panin imply that \(\GW^+\) injects into \(\GW(K)\) in this situation, with image \(\GW_\ur\).  In particular, the stalks of \(\GW_\ur\) are those of \(\GW\): \(\GW_x = (\GW_\ur)_x = \GW(\O_{X,x})\).  Consequently, the unramified filtration  has stalks
  \[
  (\urF{i}\GW)_x = (\urF{i}\GW_\ur)_x = \GW(\O_{X,x})\cap \GI^{i}(K).
  \]
  The \(\gamma\)-filtration \(\gammaF{i}\GW\) on the other hand, also viewed as a presheaf, has stalks \(\gammaF{i}\GW(\sheaf O_{X,x})\).  By \Cref{prop:local-filtration} above and Corollary~0.5 of \cite{KerzMS}, these stalks agree.
\end{proof}

Both propositions apply verbatim to the Witt ring \(\W\) in place of \(\GW\).  If, in addition to the assumptions of \Cref{prop:sheafified-filtrations}, our scheme is separated and of dimension at most three, then by \cite{BalmerWalter} the Witt presheaf \(\W\) is already a sheaf, and hence also \(\urF{i}\W\) is a filtration by sheaves.  This justifies the claim made in the introduction that the ``the unramified filtration of the Witt ring is the sheafification of the \(\gamma\)-filtration'' in this situation.

\section{Examples}\label{sec:examples}
All our examples will be smooth quasiprojective varieties over a field of characteristic different from two.
The lower-degree pieces of the filtrations on the \(\K\)-, \GrothendieckWitt and Witt rings will therefore always fit the following pattern:
\begin{align*}
  \gammaF{0}\K &= \topF{0}\K = \K              & \gammaF{0} \GW &= \GW & \gammaF{0} \W &= \W \\
  \gammaF{1}\K &= \topF{1}\K = \ker(\rank)     & \gammaF{1} \GW &= \ker(\rank) & \gammaF{1} \W &= \ker(\bar \rank) \\
  \gammaF{2}\K &= \topF{2}\K = \ker(c_1) & \gammaF{2} \GW &= \ker(w_1)  & \gammaF{2} \W &= \ker(\bar w_1)  \\
  \gammaF{3}\K &\subset \topF{3}\K =\ker(c_2)  & \gammaF{3} \GW &\subset \ker(w_2)  & \gammaF{3} \W&\subset \ker(\bar w_2)
\intertext{%
(For the topological filtration \(\topF{*}\) on the \(\K\)-ring, see \cite{Fulton:Intersection}*{Example~15.3.6}.  The symbols \(c_i\) denote the Chern classes with values in Chow groups.)  Accordingly, the first Chern class \(c_1\) and the first Stiefel-Whitney classes \(w_1\) and \(\bar w_1\) induce isomorphisms:
}
  \gr^1_\gamma\K &\cong \Pic & \gr^1_\gamma\GW &\cong H^1_\et(-,\ZZ/2) &  \gr^1_\gamma\W &\cong H^1_\et(-,\ZZ/2)
\end{align*}

Some details concerning the computations for each of the following examples are provided at the end of this section.

\begin{example}[curve]\label{eg:curve}
  Let \(C\) be a smooth curve over a field of \(2\)-cohomological dimension at most \(1\), \eg over an algebraically closed field or over a finite field. Then
  \begin{alignat*}{11}
    &\gr_\gamma^*\GW(C) &&=\gr_{\classical}^*\GW(C) &&\cong \ZZ \oplus H_\et^1(C,\ZZ/2) \oplus H_\et^2(C,\ZZ/2)\\
    &\gr_\gamma^*\W(C) &&=\gr_{\classical}^*\W(C) &&\cong \ZZ/2\oplus H_\et^1(C,\ZZ/2)
  \end{alignat*}
\end{example}

\begin{example}[surface]\label{eg:surface}
  Let \(X\) be a smooth surface over an algebraically closed field. Setting \(F_{\classical}^i\GW(X) = F_{\classical}^i\W(X) := 0\) for \(i >3\), we obtain:
  \begin{alignat*}{7}
    &\gr^*_{\classical}\GW(X)
    &&\cong \ZZ \oplus H^1_{\et}(X,\ZZ/2) \oplus H^2_{\et}(X,\ZZ/2) \oplus \CH^2(X)\\
    \gr^*_\gamma\W(X) =
    &\gr^*_{\classical}\W(X)
    &&\cong \ZZ/2\oplus H^1_\et(X,\ZZ/2) \oplus H^2_\et(X,\ZZ/2)/\Pic(X)
  \end{alignat*}
  However, in general \(\gammaF{3}\GW(X)\subsetneq \clasF{3}\GW(X)=\CH^2(X)\).  For a concrete example, consider the product \(X = C\times\PP^1\), where \(C\) is any smooth projective curve. In this case
  \begin{align*}
    \clasF{3}\GW(X) &\cong \Pic(C)\\
    \gammaF{3}\GW(X) &\cong \Pic(C)[2] &&(\text{kernel of multiplication by 2}).
  \end{align*}
\end{example}

\begin{example}[\(\PP^r\)]\label{eg:Pr}
  Let \(\PP^r\) be the \(r\)-dimensional projective space over a field~\(k\).  We first describe its \GrothendieckWitt ring. Let \(a := H_0(\lb O(1) - 1)\) and \(\rho := \lceil \frac{r}{2}\rceil\).  Then:
  \begin{align*}
    \GW(\PP^r) &\cong
    \begin{cases}
      \GW(k) \oplus \ZZ a \oplus \ZZ a^2 \oplus \dots \oplus\ZZ a^{\rho-1} \oplus \ZZ a^{\rho}\phantom{\big)} &\text{ if \(r\) is even}\\
      \GW(k) \oplus \ZZ a \oplus \ZZ a^2 \oplus \dots \oplus \ZZ a^{\rho-1} \oplus (\ZZ/2) a^{\rho} &\text{ if \(r \equiv \phantom{-}1 \mod 4\)}\\
      \GW(k) \oplus \ZZ a \oplus \ZZ a^2 \oplus \dots \oplus \ZZ a^{\rho-1} &\text{ if \(r \equiv -1 \mod 4\)}
    \end{cases}
  \end{align*}
  The multiplication is determined by the formula \(\phi\cdot a^i = \rank(\phi)a^i\) for \(\phi\in \GW(k)\) and \(i>0\), and by the vanishing of all higher powers of \(a\) (\ie \(a^i = 0\) for all \(i\geq\rho\) when \(r\equiv -1 \mod 4\); \(a^i = 0\) for all \(i>\rho\) in the other cases).\footnote{
    Over \(k=\CC\), this agrees with the ring structure of \(\KO(\CC P^n)\) as computed by Sanderson \cite{Sanderson}*{Thm~3.9}.}

  In this description, \(\gammaF{i}\GW(\PP^r)\) is the ideal generated by \(\gammaF{i}\GW(k)\) and \(a^{\lceil\frac{i}{2}\rceil}\). In particular, \(\gammaF{3}\GW(X)\) is again strictly smaller than \(\clasF{3}\GW(X)\):
  \begin{align*}
    \clasF{3}\GW(\PP^r) &= \gammaF{3}\GW(k) + (a^2,2a)\\
    \gammaF{3}\GW(\PP^r) &= \gammaF{3}\GW(k) + (a^2)
  \end{align*}
  The associated graded ring looks very similar to the ring itself:
  \begin{align*}
    \gr_\gamma^*\GW(\PP^r) & \cong
    \begin{cases}
      \gr_\gamma^*\GW(k) \oplus \ZZ a \oplus \ZZ a^2 \oplus \dots \oplus\ZZ a^{\rho-1} \oplus \ZZ a^{\rho}\phantom{\big)} &\text{ if \(r\) is even}\\
      \gr_\gamma^*\GW(k) \oplus \ZZ a \oplus \ZZ a^2 \oplus \dots \oplus \ZZ a^{\rho-1} \oplus (\ZZ/2) a^{\rho} &\text{ if \(r \equiv \phantom{-}1 \mod 4\)}\\
      \gr_\gamma^*\GW(k) \oplus \ZZ a \oplus \ZZ a^2 \oplus \dots \oplus \ZZ a^{\rho-1} &\text{ if \(r \equiv -1 \mod 4\)}
    \end{cases}
  \end{align*}
  with \(a\) of degree~2.  In the Witt ring, all the hyperbolic elements \(a^i\) vanish, so obviously \(\gr_\gamma^*\W(\PP^r) \cong \gr_\gamma^*\W(k)\).
\end{example}

\begin{example}[\(\AA^1\setminus 0\)]\label{eg:punctured-A^1}
  For the punctured affine line over a field \(k\), we have
  \begin{alignat*}{7}
    \GW(\AA^1\setminus 0) &\;\cong\; &\GW(k) &\;\oplus\;& \W(k)\red\eps& \\
    \gammaF{i}\GW(\AA^1\setminus 0) &\;\cong\; &\GI^i(k) &\;\oplus\;&\I^{i-1}(k)\red\eps&
  \end{alignat*}
  for some generator \(\red\eps\in\gammaF{1}\GW(\AA^1\setminus 0)\) satisfying \(\red\eps^2 = 2\red\eps\).  In this example, \(\gammaF{3}\GW(\AA^1\setminus 0) = \ker(w_2)\).
\end{example}

\begin{example}[\(\AA^{4n+1}\setminus 0\)]\label{eg:punctured-A^d}
  For punctured affine spaces of dimensions \(d\equiv 1 \mod 4\) with \(d > 1\), there is a similar result for the \GrothendieckWitt group \cite{BalmerGille}:
  \[
  \GW(\AA^{4n+1}\setminus 0) \cong \GW(k) \oplus \W(k)\red\eps
  \]
  for some \(\red\eps\in \gammaF{1}\GW(\AA^1\setminus 0)\).
  However, in this case \(\red\eps^2= 0\), and the \(\gamma\)-filtration is also different from the \(\gamma\)-filtration in the one-dimensional case.
  This is already apparent over the complex numbers, where we find:
  \[
    \gammaF{i}\GW(\AA_\CC^5\setminus 0) \cong \gammaF{i}\W(\AA_\CC^5\setminus 0) \cong
    \left\{
      \begin{alignedat}{7}
      \ZZ/2 &\red\eps  &&\text{ for } i=1,2\\
       0 & && \text{ for } i \geq 3
      \end{alignedat}
\right.
\]
  In particular, in this example \(\gammaF{3}\W(X)\neq \clasF{3}\W(X)\), the latter being non-zero since since \(w_2\) and \(\bar w_2\) are zero.
\end{example}

\begin{proof}[Calculations for \Cref{eg:curve} (curve)]
  Consider the summary at the beginning of this section.  In dimension~\(1\), we have \(\topF{2}\K = 0\), so \(\gammaF{2}\K = \ker(c_1) = 0\). Moreover, by \cite{Me:WCS}*{proof of Cor.~3.7}, \(w_2\) is surjective for the curves under consideration, with kernel isomorphic to the kernel of \(c_1\). So \(w_2\) is an isomorphism.  It follows that \(\gammaF{3} \GW = \gammaF{3} \W = 0\) and hence that \(\gr_\gamma^*\GW = \gr_\classical^*\GW\) and \(\gr_\gamma^*\W = \gr_\classical^* \W\).  These graded groups are computed in [loc.\ cit., Thm~3.1 and Cor.~3.7].
\end{proof}

\begin{proof}[Calculations for \Cref{eg:surface} (surface)]
  The classical filtration is computed in \cite{Me:WCS}*{Cor.~3.7/4.7}. In the case \(X=C\times\PP^1\), Walter's projective bundle formula \cite{Walter:PB}*{Thm~1.5} and the results on \(\GW^*(C)\) of \cite{Me:WCS}*{Thm~2.1/3.1} yield:
  \[
  \GW(X) \cong
  \lefteqn{\overbrace{\phantom{\ZZ \oplus \Pic(C)[2] \oplus \ZZ/2}}^{\pi^*\GW(C)}}
  \ZZ
  \oplus
  \underbrace{\Pic(C)[2]}_{H^1_\et(X,\ZZ/2)}
  \oplus
  \underbrace{
    \ZZ/2 \oplus
    \lefteqn{\overbrace{\phantom{\ZZ/2\oplus\Pic(C)}}^{\pi^*\GW^{-1}(C)\cdot\Psi}}
    \ZZ/2}_{H^2_\et(X,\ZZ/2)}
  \oplus
  \underbrace{\Pic(C)}_{\CH^2(X)}
  \]
  Here, \(\pi\colon X\twoheadrightarrow C\) is the projection and \(\Psi\in\GW^1(\PP^1)\) is a generator.   Writing \(H_i\colon \K\to\GW^i\) for the hyperbolic maps, we can describe the additive generators of \(\GW(X)\) explicitly as follows:
  \begin{compactitem}[-]
  \item \(1\) (the trivial symmetric line bundle)

  \item \(a_{\lb L} := \pi^*{\lb L}-1\), for each symmetric line bundle \(\lb L\) on \(C\), \ie for each \(\lb L\in\Pic(C)[2]\)
  \item \(b := H_0(\pi^*\lb L_1-1)\), where \(\lb L_1\) is a line bundle of degree~1 on \(C\) (hence a generator of the free summand of \(\Pic(C)\))
  \item \(c := H_{-1}(1)\cdot \Psi = H_0(F\Psi)\); here \(F\Psi = \lb O(-1) - 1\) with \(\lb O(-1)\) the pullback of the canonical line bundle on \(\PP^1\)
  \item \(d_{\lb N} := H_{-1}(\pi^*\lb N - 1) \cdot \Psi = H_0((\pi^*\lb N - 1)\cdot F\Psi)\), for each \(\lb N\in \Pic(C)\).
  \end{compactitem}
  In this list, the generators appear in the same order as the direct summands of \(\GW(X)\) that they generate appear in the formula above.
  An alternative set of generators is obtained by replacing the generators \(d_{\lb N}\) by the following generators:
  \begin{align*}
    d'_{\lb N}
    &:= H_0(\pi^*\lb N\otimes \lb O(-1) - 1)\\
    &\,=\begin{cases}
      d_{\lb N} + c &\text{ if \(\lb N\) is of even degree}\\
      d_{\lb N} + b + c &\text{ if \(\lb N\) is of odd degree}
    \end{cases}
  \end{align*}
  The only non-trivial products of the alternative generators are
  \(a_{\lb L}c = a_{\lb L}d'_{\lb N} =  d_{\lb L}' + c \; (= d_{\lb L})\).
  Moreover, the effects of the operations \(\gamma^{i}\) on the alternative generators is immediate from \Cref{lem:gamma-of-H-line} below. So \Cref{lem:gamma-filtration-generators} tells us that \(\gammaF{3}\GW\) has additive generators
  \[
  \gamma^{1}(a_{\lb L})\cdot \gamma^{2}(c)
  = a_{\lb L}\cdot (-c)
  = d_{\lb L}
  \]
  with \(\lb L\in \Pic(C)[2]\). Thus, \(\gammaF{3}\GW(X) \cong \Pic(C)[2]\), viewed as subgroup of the last summand in the formula above.  We also find that \(\gammaF{4}\GW(X) = 0\).
\end{proof}

\begin{proof}[Calculation of the ring structure on \(\GW(\PP^r)\) (\Cref{eg:Pr})]\mbox{}\\
  By \cite{Walter:PB}*{Thms~1.1 and 1.5}, the \GrothendieckWitt ring of projective space can be additively described as
  \[
  \GW(\PP^r) =
  \begin{cases}
    \GW(k) \oplus \ZZ a_1 \oplus \dots \oplus \ZZ a_{\rho} & \text{ if \(r\) is even}\\
    \GW(k) \oplus \ZZ a_1 \oplus \dots \oplus \ZZ a_{\rho-1} \oplus (\ZZ/2) H_0(F\Psi) & \text{ if \(r \equiv \phantom{-}1 \mod 4\)}\\
    \GW(k) \oplus \ZZ a_1 \oplus \dots \oplus \ZZ a_{\rho-1} & \text{ if \(r \equiv -1 \mod 4\)}
  \end{cases}
  \]
  where \(a_i = H_0(\lb O(i) - 1)\) and \(\Psi\) is a certain element in \(\GW^r(\PP^r)\). Moreover, by tracing through Walter's computations, we find that
  \begin{equation}\label{eq:Pr-FPsi-formula}
    H_0(F\Psi) = -\sum_{j=1}^{\rho}(-1)^j\binom{r+1}{\rho-j}a_j.
  \end{equation}
  Indeed, we see from the proof of \cite{Walter:PB}*{Thm~1.5} that \(F\Psi = \lb O^{\oplus N} - \lambda^\rho(\Omega)(\rho)\) in \(\K(\PP^r)\), where \(\Omega\) is the cotangent bundle of \(\PP^r\) and \(N\) is such that the virtual rank of this element is zero.

  The short exact sequence \(0\to \Omega \to \lb O^{\oplus (r+1)}(-1) \to \lb O \to 0\) over \(\PP^r\) implies that
  \[
  \lambda^\rho(\Omega) = \lambda^\rho(\lb O^{\oplus(r+1)}(-1) - 1) \text{ in } \K(\PP^r),
  \]
  from which \eqref{eq:Pr-FPsi-formula} follows by a short computation.

  An element \(\Psi \in \GW^{r}(\PP^r)\) also exists in the case \(r\equiv -1\mod 4\), and \eqref{eq:Pr-FPsi-formula} is likewise valid in this case.  However, in this case, we see from Karoubi's exact sequence
  \[
  \GW^{-1}(\PP^r) \xrightarrow{F} \K(\PP^r) \xrightarrow{H_0} \GW^0(\PP^r)
  \]
  that \(H_0(F\Psi) = 0\). We can thus rewrite the above result for the \GrothendieckWitt group as
  \[
   \GW(\PP^r) =
  \begin{cases}
    \phantom{\big(}\GW(k) \oplus \ZZ a_1 \oplus \dots \oplus \ZZ a_{\rho}\phantom{\big)} & \text{ if \(r\) is even}\\
    \big(\GW(k) \oplus \ZZ a_1 \oplus \dots \oplus \ZZ a_{\rho-1} \oplus \ZZ a_{\rho}\big)\big/2h_r & \text{ if \(r \equiv \phantom{-}1 \mod 4\)}\\
    \big(\GW(k) \oplus \ZZ a_1 \oplus \dots \oplus \ZZ a_{\rho-1} \oplus \ZZ a_{\rho}\big)\big/h_r & \text{ if \(r \equiv -1 \mod 4\)}
  \end{cases}
  \]
  with \( h_r := \sum_{j=1}^{\rho}(-1)^j\binom{r+1}{\rho-j}a_j\).

  To see that we can alternatively use powers of \(a := a_1\) as generators, it suffices to observe that for all \(k\geq 1\),
  \begin{align}
    \label{eq:Pn-a_k}
    a_k &= a^k + \left(\ctext{4cm}{a linear combination of \(a, a^2, \dots, a^{k-1}\)}\right),
    \intertext{\ignorespaces
      which follows inductively from the recursive relation
    }
    \label{eq:Pn-recursive}
    a_k &= (a + 2) a_{k-1} - a_{k-2} + 2a.
  \end{align}
    for all \(k\geq 2\). (\(a_0 := 0\).)

   Next, we show that \(a^k = 0\) for all \(k > \rho\).
   Let \(x := \sheaf O(1)\), viewed as an element of \(\K(\PP^r)\).  The relation \((x-1)^{r+1} = 0\) in \(\K(\PP^r)\) implies that
   \[
    (x-1) + (x^{-1}-1) = \sum_{i=2}^r (-1)^i(x-1)^i,
    \]
  so that we can compute:
  \begin{align*}
    a^k
    = [H(x-1)]^k
    &= H\left( \left[FH(x-1)\right]^{k-1} (x-1) \right)\\
    &= H\left( \left[(x-1) + (x^{-1} - 1)\right]^{k-1} (x-1) \right)\\
    &= H\left( (x-1)^{2k-1} +  \ctext{5cm}{higher order terms in \((x-1)\)} \right)\\
    &= 0 \quad \text{for \(2k-1 > r\), or, equivalently, for \(k > \rho\).}
  \end{align*}
  \Cref{eq:Pn-a_k} also allows us to rewrite \(h_r\) in terms of the powers of \(a\).  Inductively, we find that \(h_r = (-a)^{\rho}\) for all odd \(r\), where \(\rho = \lceil \frac{r}{2}\rceil\).
\end{proof}

\begin{proof}[Calculation of the \(\gamma\)-filtration on \(\GW(\PP^r)\) (\Cref{eg:Pr}, continued)]\mbox{}\\
  We claim above that \(\gammaF{i}\GW(\PP^r)\) is the ideal generated by \(\gammaF{i}\GW(k)\) and \(a^{\lceil\frac{i}{2}\rceil}\).  Equivalently, it is the subgroup generated by \(\gammaF{i}\GW(k)\) and by all powers \(a^j\) with \(j \geq \frac{i}{2}\).
  To verify the claim, we note that by \Cref{lem:gamma-of-H-line} below, we have \(\gamma_i(a_j) = \pm a_j\) for \(i = 1, 2\), while for all \(i > 2\) we have \(\gamma_i(a_j) = 0\).
  In particular, \(a=a_1 \in\gammaF{2}\GW(\PP^r)\), and therefore \(a^j\in\gammaF{2j}\GW(\PP^r)\).  This shows that all the above named additive generators indeed lie in \(\gammaF{i}\GW(\PP^r)\).  For the converse inclusion, we note that by \Cref{lem:gamma-filtration-generators}, \(\gammaF{i}\GW(\PP^r)\) is additively generated by \( \gammaF{i}\GW(k)\) and by all finite products of the form
  \[
  \prod_j \gamma^{i_j}(a_{\alpha_j})
  \]
  with \(\sum_j i_j\geq i\).  Such a product is non-zero only if \(i_j\in\{0,1,2\}\) for all \(j\), in which case it is of the form
  \(
  \pm \prod_j a_{\alpha_j}
  \)
  with at least \(\frac{i}{2}\) non-trivial factors.  By \eqref{eq:Pn-a_k}, each non-trivial factor \(a_{\alpha_j}\) can be expressed as a non-zero polynomial in \(a\) with no constant term.  Thus, the product itself can be rewritten as a linear combination of powers \(a^j\) with \(j\geq\frac{i}{2}\).
\end{proof}

\begin{proof}[Calculations for \Cref{eg:punctured-A^1} \((\AA^1\setminus 0)\)]\mbox{}\\
  The Witt group of the punctured affine line has the form
  \(
  \W(\AA^1\setminus 0) \cong \W(k) \oplus \W(k)\eps,
  \)
  where \(\eps = (\lb O, t)\), the trivial line bundle with the symmetric form given by multiplication with the standard coordinate (\eg \cite{BalmerGille}).  It follows that
  \[
  \GW(\AA^1\setminus 0) \cong \GW(k) \oplus \W(k)\red\eps,
  \]
  where \(\red\eps := \eps - 1\).
  As for any symmetric line bundle, \(\eps^2 = 1\) in the \GrothendieckWitt ring; equivalently, \(\red\eps^2 = -2\red\eps\).
  To compute the \(\gamma\)-filtration, we need only observe that \(\GW(\AA^1\setminus 0)\) is generated by line elements. So
  \begin{align*}
    \gammaF{i}\GW(\AA^1\setminus 0)
    &= \left(\gammaF{1}\GW(\AA^1\setminus 0)\right)^i\\
    &= \left(\GI(k) \oplus \W(k)\red\eps\right)^i\\
    &= \GI^i(k) \oplus \I^{i-1}(k)\red\eps
  \end{align*}

  The ^^e9tale cohomology of \(\AA^1\setminus 0\) has the form
  \[
  H^*_\et(\AA^1\setminus 0,\ZZ/2) \cong H^*_\et(k,\ZZ/2) \oplus H^*_\et(k,\ZZ/2)w_1\eps.
  \]
  Recall that when we write \(\ker(w_1)\) and \(\ker(w_2)\), we necessarily mean the kernels of the restrictions of \(w_1\) and \(w_2\) to \(\ker(\rank)\) and \(\ker(w_1)\), respectively.
  An arbitrary element of \(\GW(\AA^1\setminus 0)\) can be written as \(x + y\red\eps\) with \(x, y\in \GW(k)\).
  For such an element, we have \(w_1(x + y\red\eps) = w_1x + \rank(y) w_1\eps\), so the general fact that \(\ker(w_1) = \gammaF{2}\GW\) is consistent with our computation.
  When \(\rank(y) = 0\), we further find that
  \[
  w_2(x + y\red\eps) = w_2 x + w_1 y \scup w_1 \eps,
  \]
  proving the claim that \(\ker(w_2) = \gammaF{3}\GW\) in this example.
\end{proof}

\begin{proof}[Calculations for \Cref{eg:punctured-A^d} \((\AA^{4n+1}\setminus 0)\)]\mbox{}\\
  Balmer and Gille show in \cite{BalmerGille} that for \(d = 4n+1\) we have  \(\W(\AA^d\setminus 0)\cong \W(k)\oplus \W(k)\eps\) for some symmetric space \(\eps\) of even rank \(r\) such that \(\eps^2 = 0\) in the Witt ring.  Let \(\red\eps := \eps-\tfrac{r}{2}\HH\).  Then
  \[
  \GW(\AA^d\setminus 0)\cong \GW(k) \oplus \W(k)\red\eps
  \]
  with \(\red\eps^2 = 0\).  As the K-ring of \(\AA^d\setminus 0\) is trivial, \ie isomorphic to \(\ZZ\) via the rank homomorphism, \(\gammaF{i}\GW(\AA^d\setminus 0)\) maps isomorphically to \(\gammaF{i}\W(\AA^d\setminus 0)\) for all \(i>0\).
  We now switch to the complex numbers.
  Equipped with the analytic topology, \(\AA^{4n+1}_\CC\) is homotopy equivalent to the sphere \(S^{8n+1}\), so we have a comparison map \(\GW(\AA^{4n+1}_\CC\setminus 0)\to\KO(S^{8n+1})\). As the \(\lambda\)-ring structures on both sides are defined via exterior powers, this is clearly a map of \(\lambda\)-rings.
  In fact, it is an isomorphism, as we see by comparing the localization sequences for \(\AA^d_\CC \setminus 0 \mbox{\(\lefteqn{\;\;\circ}\hookrightarrow\)} \AA^d_\CC \mbox{\(\lefteqn{\;\;\;\shortmid}\hookleftarrow\)} \{0\}\), as in the proof of \cite{Me:WCCV}*{Thm~2.5}.
  The \(\lambda\)-ring structure on \(\KO(S^{8n+1})\) can be deduced from \cite{Adams:Spheres}*{Thm~7.4}:  As a special case, the theorem asserts that the projection \(\RR\PP^{8n+1}\twoheadrightarrow \RR\PP^{8n+1}/\RR\PP^{8n}\simeq S^{8n+1}\) induces the following map in \(\KO\)-theory.
  \[\xymatrix@C=0pt@R=6pt{
    {\KO(S^{8n+1})}        \ar@{^{(}->}[d] \ar@{}[r]|{\cong}
    & {\factor{\ZZ[\red\eps]}{(2\red\eps, \red\eps^2)}}
    &&&& {\red\eps}          \ar@{|->}[d]
    \\
    {\KO(\RR\PP^{8n+1})} \ar@{}[r]|{\cong}
    & {\quad\quad\quad\factor{\ZZ[\red\lambda]}{(2^f\red\lambda,\red\lambda^2 - 2\red\lambda)}}
    &&&& {2^{f-1}\red\lambda}
  }\]
    Here, \(\lambda\) is the canonical line bundle over the real projective space,  \(\red\lambda := \lambda-1\), and \(f\) is some integer.
  Thus, \(\gamma_t(2^{f-1}\red\lambda) = (1 + \red\lambda t)^{2^{f-1}}\) and we find that \(\gamma_i(\red\eps) = c_i\red\eps\) for \(c_i := \tbinom{2^{f-1}}{i}2^{i-f}\).
  Note that \(c_i\) is indeed an integer: by Kummer's theorem on binomial coefficients, we find that the highest power of two dividing \(\binom{2^{f-1}}{i}\) is at least \(f-1-k\), where \(k\) is the highest power of two such that  \(2^k \leq i\).
  In fact, modulo two we have \(c_2 \equiv 1\) and \(c_i \equiv 0\) for all \(i > 2\).  So the \(\gamma\)-filtration is as described.
\end{proof}

Finally, here is the lemma referred to multiple times above.
\begin{lem}\label{lem:gamma-of-H-line}
  Let \(\sheaf L\) be a line bundle over a scheme \(X\) over \(\ZZ[\frac{1}{2}]\). Then
  \begin{align*}
    \gamma_2(H(\sheaf L - 1)) &= -H(\sheaf L -1)
  \end{align*}
  and \( \gamma_i(H(\sheaf L - 1)) = 0 \) in \(\GW(X)\) for all \(i > 2\).
\end{lem}
\begin{proof}
Let us write \(\lambda_t(x) = 1 + xt + \lambda^2(x)t^2 + \dots \) for the total \(\lambda\)-operation, and similarly for \(\gamma_t(x)\).  Then \(\lambda_t(x+y) = \lambda_t(x)\lambda_t(y)\), \(\gamma_t(x+y)=\gamma_t(x)\gamma_t(y)\), and \(\gamma_t(x) = \lambda_{\frac{t}{1-t}}(x)\). Let \(a:= H(\sheaf L - 1)\). From
  \begin{align*}
    \lambda_t(a)
    &= \frac{\lambda_t(H\sheaf L)}{\lambda_t (H1)}
    = \frac{1 + (H\sheaf L)t + \det(H\sheaf L)t^2}{1 + (H1)t + \det(H1)t^2}
    = \frac{1 + (H\sheaf L)t + \langle -1 \rangle t^2}{1 + (H1)t + {\langle -1 \rangle} t^2}
  \intertext{we deduce that}
    \gamma_t(a)
    &= \frac{1 + (H\sheaf L - 2)t + (1 + {\langle -1 \rangle} - H\sheaf L)t^2}{1 + (H1 - 2)t + (1 + {\langle -1 \rangle} - H1)t^2} \\
    &= \frac{1 + (H\sheaf L - 2)t - H(\sheaf L - 1)t^2}{1 + (H1-2)t}\\
    &= [1 + (H\sheaf L - 2)t - H(\sheaf L - 1)t^2] \cdot \sum_{i\geq 0}(2 - H1)^i t^i.
  \end{align*}
  Here, the penultimate step uses that \(H1 \cong 1 + {\langle -1 \rangle}\) when two is invertible.

  In order to proceed, we observe that \(H1\cdot Hx = H(FH1\cdot x) = 2 Hx\) for any \(x\in \GW(X)\).
  It follows that
  \(
    (2 - H1)^i = 2^{i-1} (2 - H1)
  \)
  and hence that
  \[
  [1 + (H\sheaf L - 2)t - H(\sheaf L - 1)t^2] \cdot (2 - H1)^it^i
  = 2^{i-1}(2-H1) (1-2t)t^i
  \]
  for all \(i \geq 1\).
  This implies that the above expression for \(\gamma_t(a)\) simplifies to \(1 + H(\sheaf L - 1)t - H(\sheaf L -1)t^2\), as claimed.
\end{proof}

\end{document}